\documentclass[11pt,leqno]{article}
\usepackage{amsmath,amssymb,amsthm,graphicx,float}
\usepackage[parfill]{parskip}		% skip a line instead of indenting new paragraphs
\usepackage{caption}
\usepackage{subcaption}
\usepackage{psfrag,url}
\newcommand{\C}{\mathbb{C}}

\newcommand{\D}{\mathbb{D}}

\renewcommand{\H}{\mathbb{H}}
\newcommand{\bL}{\mathbb{L}}

\newcommand{\R}{\mathbb{R}}

\newcommand{\wC}{\widehat \C}
\newcommand{\wR}{\widehat \R}

\def\PSL{\operatorname{PSL}}
\def\SLE{\operatorname{SLE}}

\def\Mo{M\"obius}

\newcommand{\ii}{\mathfrak{i}}
\newcommand{\ee}{\mathrm{e}}
\newcommand{\dd}{\mathrm{d}}

\newcommand{\eps}{\varepsilon}

\newcommand{\g}{\gamma}
\newcommand{\Om}{\Omega}

\newcommand{\go}{\omega}
\newcommand{\om}{\omega}
\newcommand{\gO}{\Omega}
\renewcommand{\O}{\Omega}

\newcommand{\gs}{\sigma}

\renewcommand{\Im}{\textnormal{Im}}
\renewcommand{\Re}{\textnormal{Re}}
\newcommand{\gve}{\varepsilon}
\newcommand{\gvp}{\varphi}

\newcommand{\half}{\frac{1}{2}}

\newtheorem{thm}{Theorem}[section]

\newtheorem{lemma}[thm]{Lemma}
\newtheorem{cor}[thm]{Corollary}

\newtheorem{prob}[thm]{Problem}

\theoremstyle{definition}
\newtheorem{remark}[thm]{Remark}

\newtheorem{definition}[thm]{Definition}
\newtheorem{example}[thm]{Example}

\usepackage[dvipsnames]{xcolor}

\newcommand{\blue}{\textcolor{black}}

\begin{document}

\title{Piecewise geodesic Jordan curves I: weldings, explicit computations, and 
Schwarzian derivatives}
\author{Donald Marshall\footnote{University of Washington \emph{dmarshal@uw.edu}},\qquad Steffen Rohde\footnote{University of Washington \emph{rohde@uw.edu}}, \qquad  Yilin Wang\footnote{ETH Z\"urich \emph{yilin.wang@math.ethz.ch}}}
\maketitle

\begin{abstract} 
We consider Jordan curves of the form $\gamma=\cup_{j=1}^n \gamma_j$ on the Riemann sphere for which each $\gamma_j$ is a hyperbolic geodesic in $(\wC \smallsetminus \gamma)\cup \gamma_j$. These Jordan curves are characterized by their conformal welding being piecewise M\"obius. We show that the Schwarzian derivatives of the uniformizing mappings of the two regions in $\wC\smallsetminus \gamma$ form a rational function with at most second-order poles at the endpoints of $\gamma_j$ and that the poles are simple if the curve has continuous tangents.
A key tool is the explicit computation of all \blue{$C^1$} {\it geodesic pairs}, namely \blue{$C^1$} chords $\gamma=\gamma_1\cup\gamma_2$ in a simply connected domain $D$ such that $\gamma_j$ is a hyperbolic geodesic in $D\smallsetminus \gamma_{3-j}$ for both $j=1$ and $j=2$.

\end{abstract}

\section{Introduction}
% In this paper, we introduce an interesting %new
% class of Jordan curves on the Riemann sphere, characterized by a piecewise geodesic property, and begin a systematic study of their geometric function theoretic aspects, including explicit constructions and computations of the Schwarzian derivatives of the uniformization map.

We call a Jordan curve $\gamma=\cup_{j=1}^n \gamma_j\subset \wC = \mathbb C \cup\{\infty\}$  {\it piecewise geodesic} if each {\it edge} $\gamma_j$ is a hyperbolic geodesic in the simply connected domain $\wC\smallsetminus(\gamma\smallsetminus \gamma_j)$. Piecewise geodesic curves \blue{which are also $C^1$} arise naturally as minimizers of a certain functional related to Loewner theory and the Weil--Petersson metric on the universal Teichm\"uller space, and we explore these connections in our companion paper \cite{MRW}.  
In this paper, we focus on their geometric function-theoretic aspects, including explicit constructions \blue{of geodesic pairs} and computations of the Schwarzian derivatives of the uniformization map.

This type of geodesic property was motivated by a question of Wendelin Werner, and was also studied in \cite{peltola_wang} for the configuration of several disjoint chords in a simply connected domain and in \cite{bonk2021canonical} for the case of two disjoint arcs in $\wC$.

It is natural to ask if the piecewise geodesic Jordan curves consist of geodesics for the hyperbolic metric of a punctured sphere (with punctures at the endpoints of the $\g_j$, called {\it vertices}). In Lemma \ref{circle}, we show that this is only the case if $\g$ is a circle. 

A crucial role in our investigation is played by  {\it geodesic pairs}, defined in Section~\ref{sec:geodesic_pair} as chords $\gamma=\gamma_1\cup\gamma_2$ in a simply connected domain $D$ such that $\gamma_j$ is a hyperbolic geodesic in $D\smallsetminus \gamma_{3-j}$ for both $j=1$ and $j=2$.
If $\gamma=\cup_{j=1}^n \gamma_j$ is a piecewise geodesic Jordan curve, then two consecutive edges $\g_j,\g_{j+1}$ form a geodesic pair in the complement of the other edges.
Not all geodesic pairs are $C^1$ smooth (there can be logarithmic spirals at the vertex, see Example \ref{nonsmooth}).
Examples of smooth geodesic pairs first appeared in \cite{W1} as the minimizer of a certain Loewner energy among all chords passing through a given point $\zeta$.
There the description was rather indirect (as limits of $\SLE_{\kappa}$ as $\kappa\to 0$ and by means of their Loewner driving function, see also \cite{Mesik}). We 
give an explicit construction when $D=\D$ by exhibiting a conformal map of the regions $\D \smallsetminus (\gamma_1\cup\gamma_2)$ onto the upper and lower half-planes. Among other things, this allows for an explicit computation of the Schwarzian derivative of the map. 

In Section~\ref{sec:welding}, we consider a notion of conformal welding tailored to the setting of a simply connected domain bisected by a chord. A key observation is Lemma \ref{hypgeo} that characterizes hyperbolic geodesics as chords whose {\it welding map} is \Mo{} on its interval of definition. This is applied to describe the welding maps of geodesic pairs (Corollary~\ref{weldpairs}), to prove the uniqueness of $C^1$ geodesic pairs through a given point
(Theorem \ref{unique}), and again later in Section 4 to characterize piecewise geodesic Jordan curves by their welding.

Corollary \ref{pw} in Section~\ref{sec:pwcurves} states that a Jordan curve $\g$ is piecewise geodesic if and only if its welding homeomorphism is 
{\it piecewise \Mo} (see Section~\ref{sec:welding} for the definition). 
Of particular interest are  $C^1$ piecewise geodesic Jordan curves, \blue{as indicated in the first paragraph}. Theorem \ref{pwclosed} gives an explicit computation of the Schwarzian derivatives of the two Riemann maps of the two sides of $\g$ and shows that they form a rational function with simple poles at the vertices of $\gamma.$
Theorem \ref{weldthm} characterizes their welding homeomorphisms and can be viewed as a parametrization of the space of $C^1$ piecewise geodesic Jordan curves. 
We also discuss the case of non-$C^1$ piecewise geodesic Jordan curves and show that they form logarithmic spirals at each point where they are not $C^1$. 
In this case, the Schwarzian derivative of the two Riemann maps has second-order poles. In Theorem \ref{schwarzian}, we determine their coefficients in terms of the geometric {\it spiral rate} (see Definition \ref{defspiralrate}). 

Finally, in Section~\ref{sec:fourlegs}, we investigate symmetry properties of piecewise geodesic curves with four edges and prove that in the $C^1$ case, the curve is invariant under all four automorphisms of the sphere punctured in the four vertices of $\gamma.$

\bigskip 

 \noindent {\bf Acknowledgment}: We thank Mario Bonk for discussions. We also thank an anonymous referee for a careful reading and for comments that led to improvements of the manuscript.
 S.R. is supported by NSF Grants DMS-1700069, DMS‐1954674 and DMS-2350481. Y.W. is partially supported by NSF grant DMS-1953945.

\section{Geodesic pairs}  \label{sec:geodesic_pair}

A {\it Jordan curve}
%, or \emph{\it loop}, 
in $\C$ is a continuous injective image of the unit circle in $\C$. A similar
definition holds for curves in the extended plane, $\wC$, 
where continuity refers to the spherical metric in the Riemann sphere. 
If $D\subsetneq \C$ is a simply
connected region, then by the Riemann mapping theorem, there is a conformal (i.e. one-to-one and analytic) 
map $\varphi$ of the unit disk $\D$ onto $D$. If $\partial D$ is a Jordan
curve, then $\varphi$ extends to be a one-to-one and continuous map of the closed disk $\overline \D$ onto $\overline D$ by Carath\'eodory's theorem. 
If $D$ is not bounded by a Jordan curve then we can
view $\varphi$ as a one-to-one map of $\partial\D$ onto the prime ends of $\partial D \subset
\wC$. 
Henceforth, $a\in \partial D$ means that $a$ is a prime end in 
$\partial D$.

Throughout this article, $\H=\{z\in \C: \Im z >0\}$ is the upper half-plane, $\bL=\{z\in\C: \Im z
< 0\}$ is the lower half-plane and $\D^*=\{z:|z|>1\}\cup\{\infty\}$ is the complement in the
extended plane $\wC$ of $\overline \D$. We will denote the positive real
axis by $\R^+$, the positive imaginary axis by $\ii\R^+$, and the extended real line $\wR = \mathbb R \cup \{\infty\}$.

If $D$ is simply connected and if $f$ is a conformal map of $\H$ onto $D$ then $\gamma=f(\ii\R^+)$ is the {\bf hyperbolic geodesic}
from the prime end $a=f(0)$ to %the 
another
prime end $b=f(\infty)$. 
%Equivalently, the image of a diameter by aconformal map of $\D$ onto $D$ is a hyperbolic geodesic. 
Subarcs of a hyperbolic geodesic are called {\bf hyperbolic geodesic arcs}.
See Section I.4 in \cite{GM} for an equivalent definition in terms of shortest curves in the hyperbolic metric on $D$. \blue{Our definition of geodesic pairs below also makes sense in the hyperbolic geometry of regions that are not simply connected, but in this article, we restrict ourselves to simply connected regions.} We note that a hyperbolic geodesic is an analytic curve.

If $D \subsetneq \widehat\C$ and if $a,b$ are prime ends in $\partial D$, we say that $\gamma$ is a {\bf chord} in  $(D;a,b)$ if $\gamma$ has a continuous
injective parameterization $\gamma:(0,T) \to D$ such that $\gamma(t)\to a$ as $t\to 0$ and $\gamma(t) \to b$ as $t \to T$.  We allow $b=a$ in this
definition. 

\bigskip

\begin{definition}\label{gp}
Let  \blue{$D \subsetneq \wC$ be a simply connected domain,} $\zeta \in D$, and $a,b\in \partial D$. We define a {\bf geodesic pair} in $(D;a,b,\zeta)$ to be a chord $\gamma$ in $(D;a,b)$ such that $\gamma=\gamma_1\cup\gamma_2$
and $\gamma_1\cap\gamma_2=\{\zeta\}$ where $\gamma_1$ is the (hyperbolic) geodesic in $(D\smallsetminus \gamma_2 ; a,\zeta)$ and  $\gamma_2$ is the geodesic in $(D\smallsetminus
\gamma_1;\zeta,b)$. 
\end{definition}
\blue{We note that if $D = \wC \smallsetminus \{\delta\}$, with $\delta\in\wC$, then $\partial D = \{\delta\}$, hence $a = b = \delta$ in the above definition.}

\bigskip

From the definition of hyperbolic geodesics, if $\varphi$ is a conformal map defined on
$D$, then $\gamma$ is a geodesic pair in $(D;a,b,\zeta)$ if and only if $\varphi(\gamma)$ is a geodesic
pair in $(\varphi(D);\varphi(a),\varphi(b),\varphi(\zeta))$.

For orientation of the reader, we note that if $D$ is simply connected and if $\gamma$ is a
geodesic pair in $(D;a,b,\zeta)$ then $\gamma$ cannot be the union of two geodesic arcs in $D$ unless these arcs are subsets of the same geodesic in $D$. To see this, we may suppose $D=\D$
and $\zeta=0$. The
hyperbolic geodesic arcs in $\D$ from $0$ to $\partial \D$ are the radial line segments. By explicit computation, the only radial geodesic in $\D \smallsetminus
[0,1)$ is the interval $(-1,0]$. By a rotation, if $\gamma_1$ is a radial line segment, then the only
radial geodesic in $\D\smallsetminus \gamma_1$ is $-\gamma_1$.

% Similarly, a geodesic pair $\gamma$ in $(D;a,b,\zeta)$ cannot be the union of two geodesic arcs in $D\smallsetminus\{\zeta\}$ unless these arcs are subsets of the same geodesic in $D\smallsetminus\{\zeta\}$.  Again, we may suppose $D=\D$ and $\zeta=0$. The 
% hyperbolic geodesics in $\D\smallsetminus\{0\}$ from $0$ to $\partial \D$ are also the radial line segments
% because the hyperbolic geometry
% on $\D\smallsetminus \{0\}$ is the hyperbolic geometry inherited from $\H$ via the universal covering
% map of $\H$ onto $\D\smallsetminus \{0\}$, which is $\pi(z)=e^{iz}$. The hyperbolic geodesic in $\H$ from $\infty$ to $x\in \R$ is
% the vertical half line ending at $x$, whose image by $\pi$ is a radial line segment and we conclude the proof as above. 

We begin with an explicit construction of a geodesic pair in
$(\D;\ee^{\ii\theta},-\ee^{-\ii\theta},0)$.  We will see in Theorem~\ref{unique} that it is the only $C^1$ geodesic pair  in
$(\D;\ee^{\ii\theta},-\ee^{-\ii\theta},0)$. The $C^1$ regularity of the curve refers to the arclength parametrization. Equivalently, the unit tangent vector to the curve is continuous.
Similarly, we say that a curve is $C^{\alpha}$, if the arclength parametrization of the curve is $C^{k}$, where $k$ is the largest integer such that $k \le \alpha$, and its $k$-th derivative is $(\alpha - k)$-H\"older continuous.

Define  
\begin{equation}\label{Gc}  G_\theta(z)=\frac{1}{2}\left(z+\frac{1}{z}\right) - \ii c \log z,
\end{equation}
where $c=\sin\theta$, with $-\pi/2 \le \theta \le \pi/2$. 
Let 
\begin{equation}\label{Uc}
U_\theta=\{z: \Re (z) > -A ~{\rm and}~ \Im (z) > 0\}\cup\{z:  \Re (z) > A  ~{\rm and }~ \Im (z) <
0\}\cup (B,+\infty),
\end{equation}
where $A=(\pi/2)\sin \theta$ and $B=\cos \theta+\theta \sin\theta$.
Note $B \ge |A|$.
See Figure~\ref{jog}. 
As an aside, we remark that the functions $G_\theta$ can be used to compute the lift of an airfoil in classical
two-dimensional aeronautics.

\bigskip

\begin{figure}[h]
    \centering
    \includegraphics [width=.8\textwidth]{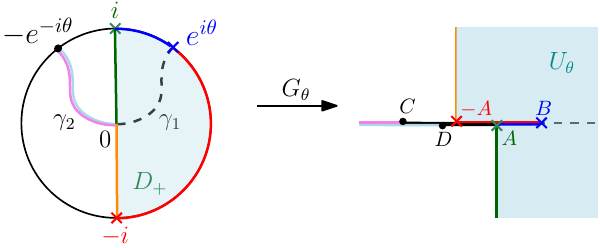}
    \caption{\label{jog} Construction of the geodesic pair $\gamma_1\cup\gamma_2$ in  $(\D;\ee^{\ii\theta},-\ee^{-\ii\theta},0)$ using the conformal map $G_\theta \colon D_+\to U_\theta$ sending $0$ to $\infty$. Here, $A =  (\pi/2) \sin \theta$, $B =  \cos \theta + \theta \sin \theta$, $C =  -A-(B+A)$, $D = A - (B-A)$, $\g_1 = G_\theta^{-1} ((B,\infty))$, and the analytic extension of $G_\theta^{-1}$ to $\mathbb C \smallsetminus (-\infty, B]$ welds $(-\infty, C)$ to $(-\infty, D)$ to form $\g_2$.}
\end{figure}

\bigskip

\begin{lemma}\label{Gcimage} The function $G_\theta$ is a conformal map of the half disk $D_+=\{z:\Re z >0 ~{\rm and}~ |z| < 1\}$ onto 
$U_\theta$, where the branch of $\log z$ is chosen
so that $-\frac{\pi}{2} \le \arg z \le \frac{\pi}{2}$. 
\end{lemma}
\begin{proof}
Fix $\theta$ and write $G=G_\theta$ and
$U=U_\theta$. Then
\begin{equation}\label{Gprime}
G^\prime(z)=\frac{z^2-2\ii cz -1}{2z^2}.%=\frac{(z-\ee^{\ii\theta})(z+\ee^{-\ii\theta})}{2z^2}
\end{equation}
 
Note that for $-1 \le y \le 1$, $y\ne 0$,
$\dd G(\ii y)/\dd y>0$, 
so as $y$ decreases from $1$ to $0$,  $\Im\, G$ decreases from $0$ to $-\infty$ with 
$\Re\, G(\ii y)=G(\ii)=A$. Similarly as $y$ decreases from $0$ to $-1$, 
$\Im\, G$ decreases from $+\infty$  to $0$ 
with $\Re\, G(\ii y)=G(-\ii)=-A$. Furthermore, 
$\dd G(\ee^{\ii t})/ \dd t = -\sin t + \sin \theta $ is positive
on $(-\pi/2, \theta)$ and negative on $(\theta,\pi/2)$.
Thus $G(\ee^{\ii t})$ increases from $G(-i)=-A$ to $B=G(\ee^{\ii\theta})=\cos\theta + \theta\sin\theta$, 
with $\Im\, G=0$, then 
decreases to
$G(\ii)=A$, again with $\Im\, G=0$. We conclude that as $z$ traces $\partial D_+$, with
positive orientation, $G(z)$
traces $\partial U$ with positive orientation. 
Indeed, for $\eps$ sufficiently small, set $D_\eps=\{z: |z| < \eps\}$.
As $z$ traces the semi--circle $\{\eps \ee^{\ii t}: \pi/2> t > -\pi/2\}$, then 
$G(\eps \ee^{\ii t})\sim
(2\eps)^{-1}\ee^{-\ii t}$ traces a curve from $G(\eps \ii)$ to $G(-\eps \ii)$. Thus if $K$ is a
compact subset of $U$, and if $\eps$ is sufficiently small then 
$G(\partial (D_+\smallsetminus D_\eps))$  winds exactly once around each point of $K$. 
By the Argument Principle, $G$ is a conformal map of $D_+$ onto
$U$.
\end{proof}

Set $\gamma_1= G^{-1}[B,+\infty)$ and $\gamma_2=-\overline \gamma_1 : =\{-\overline z: z \in \gamma\}.$
Then $\gamma_1$ is a curve in $D_+$ from $\ee^{\ii\theta}$ to $0$, and $\gamma_2\subset D_-=\D \smallsetminus D_+$  
is the reflection of $\gamma_1$ about the imaginary axis, which we parameterize as a curve 
from $0$ to $-\ee^{-\ii\theta}$. 

\bigskip
\begin{cor}\label{discpair}  The curve $\gamma=\gamma_1\cup
\gamma_2$ is a $C^1$ geodesic pair in
$(\D;\ee^{\ii\theta},-\ee^{-\ii\theta},0)$ with a horizontal tangent at $0$. When $-\pi/2<\theta<\pi/2$, $\gamma$ is perpendicular to $\partial \D$ at $\ee^{\ii\theta}$
and $-\ee^{-\ii\theta}$. When $\theta=\pi/2$, $\gamma$ is a Jordan curve with both ends at $i$, forming three angles of $\pi/3$
with $\partial \D$. When $\theta=-\pi/2$, $\gamma$ is the reflection of this latter curve about $\R$.
\end{cor}

\begin{proof}
Reflect across the vertical lines in $\partial U \smallsetminus \{0\}$ and in $\partial D_+$.
By the Schwarz reflection principle, $G$ extends to be a conformal map of $\D \smallsetminus \gamma_2$ onto
$\C \smallsetminus (-\infty, B]$. Because $F(z)=B-z^2$ is a conformal map of $\H$ onto $\C \smallsetminus
(-\infty,B]$ with $F(\ii\R^+)=(B,+\infty)$, the curve $\gamma_1= G^{-1}(F(\ii\R^+))$ is a hyperbolic geodesic
in $\D \smallsetminus \gamma_2$. Moreover $H(z)=\overline{G(-\overline z))}$ 
is a conformal map of $\D \smallsetminus \gamma_1$ onto $\C\smallsetminus (-\infty,B]$ with $H(\gamma_2)=(B,+\infty)$
and so $\gamma_2$ is a
hyperbolic geodesic in $\D \smallsetminus \gamma_1$. Thus $\gamma=\gamma_1\cup\gamma_2$ is a geodesic pair in
$(\D;\ee^{\ii\theta},-\ee^{-\ii\theta},0)$.

We also note that  
$$\arg (zG(z)) =\arg\left(\frac{1}{2}(z^2+1)-\ii cz\log z\right) \xrightarrow[]{z\to 0}
0,$$ and $\arg G =0$ along $\gamma_1$, so that
$\gamma_1$ has a horizontal tangent at $0$. \blue{We conclude by (\ref{Gprime}) that 
$$\lim_{\genfrac{}{}{0pt}{}{z\to 0}{z\in \g_1}} \arg (-G') =0,$$
so that the unit tangent to $\g_1$ is continuous at $0$. By reflection, $\g_2$ has the same property} and hence $\gamma$ is $C^1$. % has a continuous (unit) tangent at $0$. 

Finally note that by (\ref{Gc}) and (\ref{Gprime}), $G$ is analytic in a neighborhood of $\ee^{\ii\theta}$ 
and asymptotic to
$d(z-\ee^{\ii\theta})^2$, with $d\ne 0$, near $\ee^{\ii\theta}$ such that the image of
$\gamma_1\cup \partial \D$ 
is contained in $\R$. Thus $\gamma_1$ must meet $\partial \D$ at right angles. A similar
statement holds for $\gamma_2$.

When $\theta=\pi/2$, the curve $\gamma$ forms a Jordan curve beginning and ending at $i$. In this case $B=A=\pi/2$, so that $U_\theta$ has an angle of $3\pi/2$ at $B$ and $[B,\infty)$ forms an angle of $\pi$ with $[-A,A]$, or %$2/3rds$ 
two thirds of $3\pi/2$. The map $G$ is asymptotic to $K(z-\ii)^{3/2}$ near $i$.
Thus, the angle between $\gamma_1$ and $\partial \D$ at $i$ is $\pi/3=(2/3)(\pi/2)$. By reflection, the corollary follows.  
\end{proof}

\begin{remark} 
For $0<\theta<\pi/2$, $U_\theta$ (see (\ref{Uc}))  can be obtained from $U_{\pi/2}$
by removing the interval $J=[\pi/2,\cot\theta+\theta]$ and then scaling by
$\sin \theta$. If we remove the arc $\sigma=G_{\pi/2}^{-1}([\pi/2,\cot \theta
+\theta])$, together with its reflection $-\overline{\sigma}$ from $\D$ and let $f$ map $\D\smallsetminus (\sigma\cup {-\overline{\sigma}})$ onto $\D$
with $f(0)=0$ and $f'(0)>0$, then $f(\gamma_1\smallsetminus \sigma)\cup
f(\gamma_2\smallsetminus {-\overline{\sigma}})$ is the geodesic pair in $(\D;\ee^{\ii\theta},-\ee^{-\ii\theta},0)$ constructed above. Indeed, $G_\theta$ and $\sin\theta \cdot G_{\pi/2}\circ f^{-1}$ have the same image on $D^+$ and $i,0,-\ii$ have the same image as boundary points so that these two maps are the same. A similar
result holds for $-\pi/2 < \theta < 0$ using $U_{-\pi/2}$ and
$G_{-\pi/2}$. In this
sense, we can obtain all such geodesic pairs from the two Jordan curves
beginning and ending at $\pm \pi/2$, though the formulae are more
complicated because $f$ is not given explicitly. 
\end{remark}

For later use, we record some elementary computations:
We may write
\begin{equation}G^\prime(z)=\frac{z^2-2\ii cz
-1}{2z^2}=\frac{(z-\ee^{\ii\theta})(z+\ee^{-\ii\theta})}{2z^2} .
\end{equation}
By logarithmic differentiation,
the pre-Schwarzian of $G$ is
\begin{equation}\label{PG} P_G(z) := \frac{G^{\prime\prime}(z)}{G^\prime(z)}= 
\frac{-2}{z}+\frac{1}{z-\ee^{\ii\theta}} + \frac{1}{z+\ee^{-\ii\theta}}
\end{equation}
and the Schwarzian derivative of $G$ is
\begin{equation}\label{SG}S_G(z) :=
P_G^\prime(z)-\frac{1}{2}P_G^2(z)=
\frac{4\ii\sin\theta}{z}+R(z),
\end{equation}
where $R$ is rational with poles only at $\ee^{\ii\theta}$ and $-\ee^{-\ii\theta}$, and these poles have
order $2$.

\bigskip

Using the Riemann mapping theorem, we obtain easily a $C^1$ geodesic pair in other domains $(D; a,b)$ from the above construction.
If $D\ne \C$ is simply connected, $a,b \in \partial D$ and $\zeta \in D$, then let $\gamma_{a,\zeta}$ and
$\gamma_{\zeta,b}$ denote the hyperbolic geodesic arcs in $D$ from $a$ to $\zeta$ and from 
$\zeta$ to $b$, respectively.

\begin{thm}\label{Dgeopair} 
There is a geodesic pair 
$\gamma=\gamma_1\cup \gamma_2$ in $(D;a,b,\zeta)$ with continuous tangent at $\zeta$. 
The tangent to $\gamma$ at $\zeta$ bisects the angle between  $\gamma_{a,\zeta}$ and
$\gamma_{\zeta,b}$ at $\zeta$. 
If $\partial D$ has a tangent at $a$, or $b$, with $a\ne b$, then $\gamma$
meets $\partial D$ at right angles at $a$, respectively $b$. If $\partial D$ has a tangent at $a$ and $b=a$, then $\gamma$ forms three angles of $\pi/3$ with $\partial D$ at $a$.
\end{thm}

\begin{proof}
If $\varphi$ is a conformal map of
$D$ onto $\D$ with $\varphi(\zeta)=0$, then after rotating, we may suppose that
$\varphi(a)=\ee^{\ii\theta}$ and $\varphi(b)=-\ee^{-\ii\theta}$ with $-\pi/2 \le \theta \le \pi/2.$ Then
$\varphi(\gamma_{a,\zeta})$ lies on the diameter of $\D$ through $\ee^{\ii\theta}$, and 
$\varphi(\gamma_{b,\zeta})$ lies on the diameter of $\D$ through $-\ee^{-\ii\theta}$. The
geodesic pair $\sigma=\sigma_1\cup\sigma_2$ constructed in %the proof of 
Corollary \ref{discpair} has a horizontal tangent
at $0$, which forms an angle $\theta$ with each of these diameters. Set
$\gamma_j=\varphi^{-1}(\sigma_j)$, for $j=1,2$, and note that $\varphi^{-1}$ preserves angles between curves through $0$ by conformality.
If $\partial D$ has a tangent at $a$ then $\varphi^{-1}$ preserves angles between curves in
$\overline \D$ meeting at $\ee^{\ii\theta}$, by Theorem II.4.2 in \cite{GM}. Thus $\gamma_1$ meets
$\partial D$ at right angles at $a$. A similar statement holds at $b$ and in the case where $a = b$.
\end{proof}

\begin{example} {\bf $C^1$ geodesic pairs on $\H$.}
For example, if $D=\H$ and $a=0$ and $b=\infty$ and $\zeta=r \ee^{\ii
t}$, $0  < t  < \pi$, then
\begin{equation}
\tau(z)=-\ii\ee^{-\ii t}\frac{z-\zeta}{z-\overline{\zeta}}
\end{equation}
is a conformal map of $\H$ onto $\D$ with $\tau(\zeta)=0$, $\tau(0)=-\ii\ee^{\ii t}$, and
$\tau(\infty)=-\ii\ee^{-\ii t}.$ Set $\theta=t-\pi/2$, $H=G_\theta \circ
\tau$, $\sigma_1=\tau^{-1}(\gamma_1)$ and $\sigma_2=\tau^{-1}(\gamma_2)$, where
$\gamma_1\cup\gamma_2$ is the constructed geodesic pair in $(\D;\ee^{\ii\theta},-\ee^{-\ii\theta},0)$.
Then $\sigma=\sigma_1\cup \sigma_2$ is a geodesic pair in $(\H;0,\infty,\zeta)$ and $H$ is
a conformal map of $\H\smallsetminus \sigma_1$ onto $\C \smallsetminus (-\infty,B]$ where
$B=\cos\theta+\theta \sin\theta$. The
corresponding map of $\H\smallsetminus \sigma_2$ onto the slit plane is
$\overline{G(-\overline{\tau(z)})}$. 
The hyperbolic geodesic in $\H$ from $0$ to $\zeta=r\ee^{\ii t}$ lies on  the circle through $0$ and $\zeta$ which is
perpendicular to $\R$. The hyperbolic geodesic in $\H$ from $\zeta$ to $\infty$ is the vertical half-line. 
Then by Theorem \ref{Dgeopair} the geodesic pair $\sigma$ has a tangent direction that bisects the angle 
between these hyperbolic
geodesics in $\H$. By elementary geometry, the tangent direction is $\ee^{\ii t}$, so that at
$\zeta=r\ee^{\ii t}$, $\sigma$ is
orthogonal to the circle of radius $r$ centered at $0$. This geodesic pair is also orthogonal to $\R$ at $0$
and has a vertical tangent at $\infty$. 
\end{example}

Not all geodesic pairs are $C^1$. The following example is a geodesic pair which does not have a tangent at the intersection point $\zeta$. 

%\bigskip

\begin{example} {\bf A geodesic pair in $\wC\smallsetminus\{0\}$ with logarithmic spirals.}\label{nonsmooth}

Let $S$ be the strip 
$\{z: -\pi < \Im z < \pi\}$.  
The image of $\R$, by the conformal map 
$\ii \ee^{z/2}$ of $S$ onto $\H$, 
is the positive imaginary axis, so that $\R$ is a hyperbolic geodesic in $S$.
We consider another conformal image of $S$ found by first applying the map
$z \mapsto Az$ where $A$ is chosen so that $S$ is  rotated by an angle
$\theta$ then scaled so that the image by the map $z\mapsto Az$ 
of the two parallel lines in $\partial S$ differ by $2\pi \ii$. 
More concretely, fix $\theta$ with
$-\pi/2 < \theta < \pi/2$ and set $A=1/(1-\ii\tan\theta)$ then 
$A(\pi \tan \theta + \pi \ii)=\pi \ii$ and $A(-\pi \tan\theta -\pi \ii)=-\pi \ii$. 
The function $\ee^{Az}$ maps 
the line $\{x+\ii y: y=\pi\}\subset \partial S$ onto a
logarithmic spiral $S_\theta$ which spirals from $0$ to $\infty$ as
$x$ increases. The line
$\{x+\ii y:y=-\pi\}\subset \partial S$ 
is also mapped to
exactly the same spiral. 
The image of $\R$ is the spiral $-S_\theta$ from $0$ to $\infty$. 
Because
$\ee^{Az}$ is a conformal map on $S$, $-S_\theta$ is a hyperbolic
geodesic in $\C\smallsetminus S_\theta$. The map $\ee^{Az}$ maps the region 
$S+\pi \ii$ onto $\C \smallsetminus (-S_\theta)$. Similarly $S_\theta$ is a
hyperbolic geodesic in  $\C \smallsetminus -S_\theta$. If we reparameterize $-S_\theta$ so that it is a curve
from $\infty$ to $0$, then $\gamma=
S_\theta \cup (-S_\theta)$ is a geodesic pair in $(\wC\smallsetminus\{0\};0,0,\infty)$.

\end{example}

\section{Welding Maps}\label{sec:welding}
\label{weldingmaps}

In this section, we characterize piecewise geodesic curves using the welding maps.

Suppose \blue{$\gO\subsetneq \wC$} is a simply connected region contained in the extended plane $\wC$ and 
suppose $\gamma$ is a chord in $\gO$ from $a$ to $b$, with  $a,b\in\partial \Om$.
By the Jordan curve theorem, $\gO\smallsetminus \gamma$ consists of two simply
connected regions $\gO_+$ and $\gO_-$. We choose $\gO_+$ to be the region that lies to the left of the oriented curve $\gamma$, and choose $\gO_-$ to be the region that lies to the right of $\gamma$.
Let $f_+$ and $f_-$ be
conformal maps of $\gO_+$ and $\gO_-$ onto the upper half-plane $\H$
and the lower half-plane $\bL$, respectively. By Caratheodory's
theorem, $f_+$ and $f_-$ extend to be homeomorphisms of $\gO_+\cup \gamma$ and $\gO_-\cup \gamma$ onto $\H\cup f_+(\gamma)$
and $\bL\cup f_-(\gamma)$.
The function $\go=f_- \circ {f_+}^{-1}$ is called a {\bf welding map}
associated with the curve $\gamma$ with respect to the region $\gO$.
For $\zeta\in \gamma$, $x=f_+(\zeta)$ if and only if
$\go(x)=f_-(\zeta)$. So we can view  $f_+^{-1}$ and $f_-^{-1}$ as
mapping $\H$ and $\bL$ onto $\gO_+$ and $\gO_-$ pasting together, or welding,
the intervals $f_+(\gamma)\subset \partial \H$ and $f_-(\gamma)\subset
\partial \bL$, according to the prescription $\go$. 

The function $\go$ is
an increasing homeomorphism of $f_+(\gamma)$ onto $f_-(\gamma)$. 
Here, we extend the notion of an increasing homeomorphism on an interval in $\R$ to 
functions $u$ defined on an interval $J^*\subset \wR$, the extended
real line, where we
allow $\infty \in J^*$ and $\infty \in f(J^*)$. 
An increasing homeomorphism on $J^*$ is a continuous,
injective function $u$ defined on $J^*$ which is increasing on each subinterval of $J^*$ which contains
neither  $\infty$ nor $u^{-1}(\infty)$. Our choice for
$\gO_+$ and $\gO_-$ means that $f_+$ and $f_-$ are also increasing
functions of the parameterization of $\gamma$.

\bigskip

\begin{example}\label{2weldings}{\bf Welding maps for geodesic pairs in Corollary \ref{discpair}.}

Suppose $\gamma_1, \gamma_2 \subset \D$ are defined as in Corollary \ref{discpair}. Then $G_\theta$ provides a map of the two
regions in  
$\D \smallsetminus (\gamma_1\cup \gamma_2)$ onto $\H$ and $\bL$. See Figure~\ref{jog}. The
associated welding map $\go$ is 
\begin{equation}\label{eq:pairweld}
\go(x)=\begin{cases}
x & \text{for}~ x > B =  \cos\theta+\theta\sin\theta,\cr
x + 2\pi \sin \theta & \text{for}~ x < C = -\cos\theta-(\theta+\pi)\sin\theta.\cr
\end{cases}
\end{equation}
 The expression for $\go$ on $(B,\infty)$ follows from the fact that
$G_\theta$ is analytic across $\gamma_1$. 
The welding for $x \in (-\infty, C)$ is
given by first reflecting about the vertical line
$x=-(\pi/2)\sin\theta$
then reflecting about the vertical line $x=(\pi/2)\sin\theta$. Two 
reflections form a shift (and this is why we considered regions of
the form $U_\theta$). Indeed, the first reflection is given by
$z \mapsto -\overline{z}-\pi \sin \theta$, and the second reflection is given by
$z  \mapsto -\overline{z}+\pi \sin \theta$, so that 
$\go(x)=x+2\pi\sin \theta$, mapping the interval $(-\infty,C)$ onto $(-\infty, D)$.
\end{example}

\bigskip

\begin{example} \label{weldnonsmooth}{\bf Welding maps for Example \ref{nonsmooth}.}

This case was explored in \cite{KNS}, with slightly different notation. For completeness, we give
the construction here.
Recall that the function $f(z)=\ee^{A\log z}$, with $-\pi < \arg z < \pi$ is a conformal map of $\H$ and $\bL$ onto
the two regions in $\C\smallsetminus S_\theta\cup (-S_\theta)$, where $A = 1/(1- \ii\tan \theta)$. This yields $f_+(z)= f^{-1}(z)$ for $z\in f(\H)$ 
and $f_-(z)=f^{-1}(z)$ for $z\in f(\bL)$. The welding map $\go=f_-\circ f_+^{-1}$ then satisfies
\begin{equation}\label{spiralweld}
\go(x)=\begin{cases}
x & \text{for}~ x > 0,\cr
ax & \text{for}~ x < 0,\cr
\end{cases}
\end{equation}
where $a=\exp(-2\pi\tan\theta)$. Indeed, because $f$ is analytic on $\H\cup\bL\cup \R^+$,
$\go(x)=x$ for $x >0$.
If $x\in \partial \H$, $x<0$ and $a x \in \partial \bL$ , then $\log x =\log |x|+\ii\pi$ and 
$\log ax= \log |ax| -\ii \pi $. An easy computation shows that $f(ax)=f(x)$, so that $\go(x)=ax$
for $x < 0$. 
\end{example}

 \bigskip

Non-constant maps of the form
$\gs(z)=\dfrac{az+b}{cz+d}$, where $a,b,c,d$ are complex numbers, are
called {\bf \Mo{} maps}. Note that $\gs$ is an increasing
homeomorphism of $\wR$ onto $\wR$ if and only if  $\gs(\H)=\H$, and
this occurs if and only if $\gs(\bL)=\bL$.  

We can identify the \Mo{} maps that preserve $\H$ with the group
\begin{equation}
\PSL(2,\R) = \Big\{A = \begin{pmatrix}
a & b \\ c & d 
\end{pmatrix} :  a,b,c,d \in \R, \, ad - bc = 1\Big\}_{/A \sim -A}
\end{equation}
which acts on $\H$ (and $\bL$) by 
$A : z \mapsto \dfrac{az + b}{cz + d}$.
The set of \Mo{} maps of $\wC$  is similarly 
identified with the group $\PSL(2,\C)$.

If $\go$ is an increasing homeomorphism defined on an interval $J\subset \wR$ and if
 $\gs_1$, $\gs_2$ are \Mo{} maps with $\gs_j(\H) =\H$,  $j=1,2$, 
then $\go_1= \gs_1\circ \go\circ 
\gs_2$ is also an increasing homemorphism, defined on $\gs_2^{-1}(J)$.

\bigskip 

\begin{definition}\label{equiv}
We say that two increasing homeomorphisms $\go_1, \go_2$ defined on intervals in $\wR$ are {\bf equivalent}
if  $\go_1= \gs_1\circ \go_2\circ \gs_2$, for some \Mo{} maps with $\gs_j(\H) =\H$ (and hence
$\gs_j(\bL)=\bL$), $j=1,2$. 
\end{definition}

If $\go=f_- \circ {f_+}^{-1}$ is a welding map associated with a curve $\gamma$ with respect to a region $\gO$
then every increasing homeomorphism equivalent to $\go$ is also associated with the same curve $\gamma$ in the region $\gO$.
Indeed, if $g_+=\gs_2^{-1} \circ f_+$ and $g_-=\gs_1 \circ f_-$, then
$g_-\circ g_+^{-1} = \gs_1 \circ \go \circ \gs_2.$ By the Riemann mapping
theorem, every welding associated with the curve $\gamma$ is
equivalent to $\om$. We remark that this notion of equivalence depends on
the orientation of the curve $\gamma$. If $\go$ is a welding map
associated to $\gamma\subset \gO$, then $-\go^{-1}(-x)=g_-\circ g_+^{-1}$, where $g_+=-f_-$ and
$g_-=-f_+$, is also an
increasing homeomorphism, not equivalent to $\go$, but 
associated with $\gamma$ traced
in the opposite direction.

\bigskip

\begin{lemma}\label{hypgeo} Suppose \blue{$\gO\subsetneq \C$}  is a simply connected region and suppose $\gamma\subset \gO$  is a chord in $\gO$ connecting $a \neq b \in \partial \gO$. 
%If $\Omega \ne \C$, 
Let $\go=f_- \circ {f_+}^{-1}$ be a welding map associated with $\gamma$ with respect to $\gO$. 
Then $\gamma$ is a
hyperbolic geodesic in $\gO$ 
if and only if  $\go$ is a \Mo{} map on its interval of definition.
%If $\Omega=\C$, then $\g$ is a straight line if and only if $\go$ is a \Mo{} map.
\end{lemma}

\begin{proof}  If $\gamma$ is a hyperbolic geodesic from $a$ to $b$ with $a,b\in \partial \gO$, then let
$g$ be a conformal map of $\gO$ onto $\H$ such that
$g(\gamma)=\ii\R^+$, so that $g(a)=+\infty$ and $g(b)=0$. Then $g^2\vert_{\gO^+}$ and  $g^2\vert_{\gO^-}$ are
conformal maps of the two regions $\gO_+$ and $\gO_-$, with
$\gO\smallsetminus \gamma=\gO_+\cup\gO_-$, onto $\H$ and $\bL$ respectively.
The welding map $\go_0$ on $(-\infty,0)$, in this case, is the identity map. 
But $\go=\gs_1\circ \go_0\circ \gs_2$ for some \Mo{} maps $\gs_1, \gs_2$ by the Riemann mapping theorem.
Thus $\go=\gs_1\circ \gs_2$, a \Mo{} map, on $\gs_2^{-1}(-\infty,0)$.

Conversely %suppose $\Om\ne\C$ and 
suppose
$\go$ is a \Mo{} map on its interval of definition $I$.
Let $\gs_1$ be a \Mo{} map with $\gs_1(\go(I))=(-\infty,0)$ and let $\gs_2$ be a
\Mo{} map with $\gs_2(I)=(-\infty,0)$. We can also arrange that
$\gs_j(\H)=\H$ and $\gs_j(\bL)=\bL$, $j=1,2$. 
Then if $c$ is a positive constant,
$\go_1=c\gs_1\circ \go\circ \gs_2^{-1}$ fixes $0$ and $\infty$, so we can choose $c$ so that $\go_1$
is the identity map on
$(-\infty,0)$ . Set $g_+=\sigma_2\circ f_+$ and $g_-=c\sigma_1\circ
f_-$. Then $\go_1=g_- \circ {g_+}^{-1}$ is a welding map associated
with $\gamma$. Moreover $g_+^{-1}=g_-^{-1}$ on $(-\infty,0)$ so that the function $h$ defined by 
$h=g_+^{-1}$ on $\H$ and $h= g_-^{-1}$ on $\bL$ extends to be conformal on $\H \cup \bL
\cup(-\infty,0)$.
Then $k(z)=h(z^2)$ is a conformal map of $\H$ onto $\gO$ with $k(\ii\R^+)=\gamma$.
%If $\Omega=\C$ and if $\g$ is a line, then $\omega$ is \Mo{} map by explicit computation. Conversely if $\Om=\C$ and if $\go$ is a \Mo{} map, then $\go$ is equivalent to the identity map on $\R$. We may suppose that $f_+(\infty)=f_-(\infty)=\infty$ so that the function given by $f_+$ on $\Om_+$ and $f_-$ on $\Om_-$ extends to be entire and one-to-one and hence linear. Thus $\g$ is a line.
\end{proof}

\bigskip

\begin{cor}\label{weldpairs} A welding map for a geodesic pair in $(D;a,b,\zeta)$,  \blue{with $D \subsetneq \wC$,}
is equivalent to a
welding map given by 
\begin{equation}\label{pwMo}\go(x) =
\begin{cases}
a_1 x + b_1, &\text{for}~ x > c_1\cr
a_2 x +b _2, &\text{for}~ x < c_2\cr
\end{cases}
\end{equation}
where $a_1 > 0$, $a_2 > 0$, 
$b_1, b_2, c_1, c_2 \in \R$, and $c_2\le c_1$, and $a_2c_2+b_2 \le a_1c_1+b_1$.
\end{cor}

\begin{proof}

By Lemma \ref{hypgeo}, 
the welding map $\go$ for a geodesic pair in $(D;a,b,\zeta)$
is given by \Mo{} maps on two adjacent intervals. By pre- and post-
composition with increasing \Mo{} maps, we may suppose that
the common point of these two intervals in $\wR$ is $\infty$, and
$\go(\infty)=\infty$. A \Mo{}  map which fixes $\infty$ and
is increasing as a map $\R$ to $\R$ must be of the form $ax+b$ with
$a>0$ and $b\in \R$. The lemma follows.
\end{proof}

\bigskip

If an increasing homeomorphism 
is of the form (\ref{pwMo}), we say that it is piecewise
\Mo{} near $\infty$. If it also 
satisfies $a_1=a_2$ then we say that $\go$ has a
continuous derivative equal to $a_1$ at $\infty$.

\bigskip

\begin{thm}\label{uniqconstr}
If $\go$ is a piecewise \Mo{} map near $\infty$ which 
has a continuous derivative at $\infty$ 
then $\go$ is a welding
map associated with exactly one of the geodesic pairs constructed in %the proof of
Corollary \ref{discpair}, unless $\go$ is a linear map on all of $\R$.
\end{thm}

\bigskip

\begin{proof}
Suppose $\go(x)=ax+b_1$, for $x > c_1$, and $\go(x)=ax+b_2$ for $
x < c_2$, where $c_2 \le c_1$ and $a > 0$.

If $b_1\ne b_2$ then
set
$$\varphi(z)=\left(\frac{|b_1-b_2|}{a}\right)z +c_1$$
and 
$$\tau(z)=\frac{z -(ac_1+b_1)}{|b_1-b_2|}.$$
It is elementary to check the following:
If $b_2 > b_1$ then
\begin{equation}\label{p1shift}
\go_1(x)\equiv\tau\circ\go\circ\varphi(x)=
\begin{cases}
x,&\text{for}~ x > 0\cr
x+1&\text{for}~ x < \dfrac{(c_2-c_1)a}{b_2-b_1} \le -1.\cr
\end{cases}
\end{equation}
The last inequality in (\ref{p1shift})
holds because $\go(c_2)\le \go(c_1)$.

If $b_1 > b_2$ then
\begin{equation}\label{m1shift}
\go_1(x)\equiv\tau\circ\go\circ\varphi(x)=
\begin{cases}
x,&\text{for}~ x > 0\cr
x-1&\text{for}~ x < \dfrac{(c_2-c_1)a}{b_1-b_2} \le 0.
\end{cases}
\end{equation}

If $b_1=b_2$ and $c_2 < c_1$ then set 
$$\varphi(z)= \frac{c_1-c_2}{2}x +\frac{c_1+c_2}{2}$$
and
$$\tau(z)=\frac{2x}{a(c_1-c_2)}-\frac{c_1+c_2+2b_1/a}{c_1-c_2}.$$
Then 
\begin{equation}\label{noshift}
\go_1(x)\equiv\tau\circ\go\circ\varphi (x) =x ~\text{for}~ x > 1 ~\text{and}~ 
x< -1.
\end{equation}

If $b_1=b_2$ and $c_1=c_2$ and if $\go$ is a welding for $\D\smallsetminus \gamma= D_+\cup D_-$
then $f_+^{-1}(x)=f_-^{-1}(ax+b_1)$ for $x \ne c_1$. This implies that the function $F$ given by
$f_+^{-1}$ on $\H$ and $f_-^{-1}(ax+b_1)$ on $\bL$ extends to be an entire function mapping onto the disk. 
But then $F$ is constant, which is a contradiction. 
We remark that these linear maps on all of $\R$ are welding maps for circles and lines in $\wC$, but not for geodesic pairs \blue{in $\D$}.

Apply the above normalization to the welding maps in Example \ref{2weldings}. 
So  $a=1$, $b_1=0$,
$b_2=2\pi\sin\theta$
$c_1=\cos\theta +\theta\sin\theta$ and
$c_2=-\cos\theta-(\theta+\pi)\sin\theta$.

If $\sin \theta>0$ then (\ref{p1shift}) applies with 
$\dfrac{(c_2-c_1)a}{b_2-b_1} =
-\left(\dfrac{\cot\theta+\theta}{\pi}\right)-\dfrac {1}{2}$.

If $\sin \theta<0$ then (\ref{m1shift}) applies with 
$\dfrac{(c_2-c_1)a}{b_1-b_2} =
\left(\dfrac{\cot\theta+\theta}{\pi}\right)+\dfrac {1}{2}$.

If $\sin \theta=0$, then $\theta=0$ and (\ref{noshift}) applies.

Note that  $-(\cot\theta+\theta)/\pi-1/2$
has positive derivatives and increases from 
$-\infty$ to $-1$ as $\theta$ increases from $0$ to $\pi/2$. 
Similarly $(\cot\theta+\theta)/\pi+1/2$
decreases from $0$ to $-\infty$ as $\theta$ increases from $-\pi/2$ to $0$. 
Thus the examples in %the proof of
Corollary \ref{discpair} have
a welding map in each equivalence class of piecewise \Mo{}
 maps with continuous derivative at $\infty$. 

Suppose $\gamma_\alpha=\gamma_{\alpha,1}\cup \gamma_{\alpha,2}$ and 
$\gamma_\beta=\gamma_{\beta,1}\cup\gamma_{\beta,2}$ are geodesic pairs created in %the proof of
Corollary \ref{discpair} for $\theta=\alpha$ and $\theta=\beta$. If their welding maps are equivalent, then there is a conformal map $\varphi$ of $\D$ onto $\D$  with $\varphi(\gamma_\alpha)=\gamma_\beta$, by the comment after Definition \ref{equiv}. But then $\varphi$ 
must be a \Mo{} map with $\varphi(0)=0$, and hence $\varphi$ is a rotation. But if $\gamma_\beta$ is a rotation of $\gamma_\alpha$, then $\alpha=\beta$. This proves that there is exactly one welding map from the examples in Corollary \ref{discpair}
in each equivalence class of piecewise \Mo{} 
maps with continuous derivative at $\infty$. This latter fact can also be proved by showing that no two maps of the form (\ref{p1shift}), (\ref{m1shift}), or (\ref{noshift}) are equivalent.
\end{proof}

\bigskip

If $\gvp$ is a conformal map of $\gO$ onto another simply connected region $\gvp(\gO)$ 
and if $\go$ is a welding map associated with $\gamma\subset\gO$,
then $\go$ is also a welding map associated with $\gvp(\gamma)\subset \gvp(\gO)$.
The next lemma, which is well-known, says that a piecewise analytic Jordan curve is determined by 
its welding map, up to conformal
equivalence.  This lemma and the subsequent corollary are also true 
under the less restrictive assumption that
$\gamma_1$ is a quasi-arc. 
Because the audience for this paper may come from other fields, we decided
to keep the proofs as simple as possible.

\bigskip

\begin{lemma}\label{global}
Suppose $\go$ is a welding map for both chords $\gamma_1 \subset \gO_1$ and
$\gamma_2\subset \gO_2$.
If $\gamma_1$ is piecewise analytic, then there is a conformal map $F$ of $\gO_1$  onto $\gO_2$ such that
$F(\gamma_1)=\gamma_2$.
\end{lemma}

\begin{proof} Let $f_+, f_-$ be the conformal maps onto $\H$ and $\bL$ of the two regions determined by
$\gamma_1$, and let 
 $g_+, g_-$ be the conformal maps onto $\H$ and $\bL$ of the two regions determined by $\gamma_2$.
The functions 
$F_1={g_+}^{-1} \circ f_+$ and $F_2={g_-}^{-1} \circ f_-$ are conformal maps of the regions determined 
by $\gamma_1$ onto the regions determined by $\gamma_2$.  By assumption $f_-\circ {f_+}^{-1}=\go=g_-\circ {g_+}^{-1}$ 
on $f_+(\gamma_1)$, so that $F_1=F_2$ on $\gamma_1$. Because $\gamma_1$ is piecewise analytic, the
function $F_1$ extends to be analytic on $\gO_1$ except for the isolated points where $\gamma_1$ is
not analytic. Isolated
points are removable for continuous analytic functions, so 
the function $F_1$ extends to
be a conformal map of $\gO_1$ onto $\gO_2$.
\end{proof}

It is not true that a welding $\go$ for a curve $\gamma$ with respect to a region
$\gO$ is also a welding with respect to a subregion $\gO_1\subset \gO$ for $\gamma\cap
\gO_1$, but we can cut $\gO$ along $\gamma$ to give the following local result.

Suppose $\go_1, \go_2$ are weldings for piecewise analytic chords $\gamma_1, \gamma_2$ with respect to regions
$\gO_1, \gO_2$, respectively, and suppose $f_+, f_-$ and $g_+, g_-$ are
conformal maps of the corresponding regions $(\gO_j)_+$ and $(\gO_j)_-$ onto
$\H$ and $\bL$.
 If $\go_1\vert_I=\go_2\vert_I$ for some interval
$I \subset f_+(\gamma_1)\cap g_+(\gamma_2)$, then $\go_1\vert_I$ is a welding map associated
with $(\gO_1\smallsetminus \gamma_1)\cup f_+^{-1}(I)$. Similarly,
$\go_2\vert_I$ is a welding map associated 
with $(\gO_2\smallsetminus \gamma_2)\cup g_+^{-1}(I)$. 
Lemma \ref{global} implies the following result.

\begin{cor}\label{local} 
 If $\go_1\vert_I=\go_2\vert_I$ for some interval
$I \subset f_+(\gamma_1)\cap g_+(\gamma_2)$, then there is a conformal map from $(\gO_1\smallsetminus \gamma_1)\cup f_+^{-1}(I)$ onto
 $(\gO_2\smallsetminus \gamma_2)\cup g_+^{-1}(I)$ mapping $f_+^{-1}(I) \subset \gamma_1$
onto $g_+^{-1}(I)\subset \gamma_2$.
\end{cor}

In other words, the local behavior of $\gamma_1$ is, up to local conformal equivalence, 
determined by the local behavior of $\go_1$, as was observed in \cite{KNS}.

A Jordan curve  $\gamma\subset \wC$ can be viewed as a chord in $\gO=\wC\smallsetminus \{\zeta_0\}$ with endpoints tending to 
$\partial \gO$, where $\zeta_0 \in \gamma$.
In particular, if $\gamma_1$ and $\gamma_2$ are 
Jordan curves associated with the same welding map $\go$, and if $\gamma_1$ is piecewise analytic,
then by  Corollary \ref{local}
there is a \Mo{} map
$\tau$ such that $\gamma_2=\tau(\gamma_1)$.  Thus increasing piecewise analytic homeomorphisms of $\wR$, 
modulo pre and post composition with \Mo{} 
maps, correspond to piecewise analytic Jordan curves in $\wC$ modulo post composition with \Mo{} maps. 

Corollary \ref{local} also yields the following uniqueness result. 

\bigskip

\begin{thm}\label{unique} If \blue{$D\subsetneq \C$} is simply connected, if $a,b\in \partial D$, and if $\zeta\in D$, 
then there is a unique $C^1$
geodesic pair $\gamma$ in $(D;a,b,\zeta)$. % with continuous (unit) tangent at $\zeta$. 
There is a conformal map
$\varphi$ of $D$ onto $\D$ so that $\varphi(\gamma)$ is one of the geodesic pairs constructed in 
%the proof of 
Corollary \ref{discpair}. \blue{If $D=\C$ and if $\g$ is a geodesic pair in $(D;\infty,\infty,\zeta)$ then $\g$ is either a logarithmic spiral about $\zeta$ or a straight line through $\zeta$.}

\end{thm}

%\bigskip

\begin{proof} \blue{If $D\ne\C$,} by applying a conformal map of $D$ onto $\D$, we may suppose $D=\D$ and $\zeta=0$. Let $\gamma=\gamma_1\cup\gamma_2$ be a
geodesic pair in $(\D;a,b,0)$. \blue{We may also assume that
$f_+(0)=f_-(0)=\infty$, so that the corresponding welding map $\go$ is defined on an interval $J\subset \wR$ that contains} $\infty$ with $\go(\infty)=\infty$. 
Then
\begin{equation}\label{twoL}
\go(x)=
\begin{cases}
a_1 x +b_1 & \text{for}~ x>c_1\cr
a_2 x + b_2 & \text{for}~ x < c_2\cr
\end{cases}
\end{equation}
where $a_1, a_2 > 0$, $b_1,b_2,c_1,c_2 \in \R$, with $c_2 \le c_1$.
If $a_1 \ne a_2$, set $\varphi(x)=x/a_1 +(b_2-b_1)/(a_1-a_2)$ and
$\tau(x)=x+(a_2 b_1-a_1 b_2)/(a_1-a_2)$.
Then
\begin{equation}\label{spiralw}
\go_1(x)\equiv \tau\circ\go\circ\varphi(x)=
\begin{cases}
x &\text{for}~ x > \varphi^{-1}(c_1)\cr
\dfrac{a_2}{a_1}x &\text{for}~ x < \varphi^{-1}(c_2)\cr
\end{cases}
\end{equation}
Choose $\theta$ so that $a_2/a_1 =\exp(-2\pi \tan \theta)$ with
$-\pi/2 < \theta < \pi/2$. By Example \ref{nonsmooth}, for $x$ sufficiently large, the welding map
$\go_1(x)$ agrees with a welding map associated with a logarithmic spiral $S_\theta$. By Corollary \ref{local}
there is a conformal map $F$ defined in a neighborhood $U$ of $0$ with $F(S_\theta\cap U)=\gamma\cap
F(U)$. In particular, $\gamma$  does not have a continuous (unit) tangent at $0$.  So if $\gamma$ has a continuous
(unit) tangent at $0$, then it has a welding map with continuous derivative at $\infty$. By 
Theorem \ref{uniqconstr}, the welding map $\go$ for $\gamma$ in
$\D$ is associated with exactly one of the geodesic pairs constructed in %the proof of
Corollary \ref{discpair}.
By Lemma \ref{global} there is a conformal map $H$ of $\D$ onto $\D$ so that $H(\gamma)$ is one of these constructed
geodesic pairs. Because $H$ is a \Mo{} map fixing the origin, $H$ must be a rotation. There is exactly one
rotation of $\D$ and one $\theta$, $-\pi/2\le  \theta \le \pi/2$, so that $a,b$ are rotated to
$\ee^{\ii\theta}$ and $-\ee^{-\ii\theta}$.

\blue{If $D=\C$, and if $\g=\g_1\cup\g_2$ is a geodesic pair in $(D,\infty,\infty,\zeta)$, then the corresponding welding map must map $\wR$ onto $\wR$ and hence in (\ref{twoL}) we must have $c_1=c_2$ and $a_1 c_1+b_1=a_2 c_2 +b_2$. If $a_1\ne a_2$ then as in (\ref{spiralw}), $0=\varphi^{-1}(c_j)$ and $\go$
is conjugate to $\go_a(x)=x$ for $x > 0$ and 
$\go_a(x)=ax$ for $x < 0$, $a \ne 1$. This of course gives the logarithmic
spirals
discussed earlier.  If $a_1=a_2$ then $b_1=b_2$. Setting $\varphi(x)= x/a_1 -b_1/a_1$ and $\tau(x)=x$, we conclude $\go$ is conjugate to $\go_1(x)=x$ for $x\ne 0$. In this case, $\g$ is 
a straight line through $\zeta$.}
\end{proof}

\bigskip

The proof of Theorem \ref{unique} also gives the following information about geodesic pairs which do not
have a continuous tangent.

\bigskip

\begin{cor}\label{Dlocal}
If \blue{$D\subsetneq\wC$} is simply connected and if
$\gamma \in (D;a,b,\zeta)$ is a geodesic pair which does not have a continuous unit tangent at $\zeta$,
then $\gamma$ is asymptotically a logarithmic spiral near $\zeta$.
\end{cor}

\bigskip

 The {\it Schwarzian derivative} of an analytic function $f$ is given by
\begin{equation}
    S_f(z)=\left(\frac{f^{\prime\prime}(z)}{f^\prime(z)}\right)^\prime -\half \left(\frac{f^{\prime\prime}(z)}{f^\prime(z)}\right)^2.
\end{equation}
The composition rule for the Schwarzian derivative is
\begin{equation}\label{Schwarzcomp}
S_{g\circ h}(z)= S_g(h(z))\left(h^\prime(z)\right)^2 +S_h(z).
\end{equation}
The Schwarzian derivative of a \Mo{} map is zero. If $S_f=S_g$, then there is a \Mo{} map $\sigma$
such that $f=\sigma\circ g$. 

Our construction also gives some information about the Schwarzian derivative of the conformal map $f_\pm$ in the next corollary.
%\bigskip

\begin{cor}\label{SDpair} If $D\ne \C$ is simply connected and if $\gamma=\gamma_1\cup\gamma_2$ is a 
$C^1$ geodesic pair in $(D;a,b,\zeta)$, % with continuous (unit) tangent at $\zeta$, 
then the Schwarzian derivative $S_f$ of the conformal map $f_{\pm}:D\smallsetminus \gamma_1\cup
\gamma_2\to
\H\cup \bL $ extends
to be analytic on $D$, except for a simple pole at $\zeta$. 
The residue of the pole of $S_f$ at $\zeta$ is 
$(4\ii \sin \theta)h^\prime(\zeta)$, where $h$ is the conformal map of
$D$ onto $\D$ with $h(\zeta)=0$ and $h(a)=\ee^{\ii\theta}$
and $h(b)=-\ee^{-\ii\theta}$. 

\end{cor}

\begin{proof}
Conformal maps have non-vanishing derivative, so that $S_g$ is analytic whenever
$g$ is conformal.
The composition rule for
Schwarzian derivatives (\ref{Schwarzcomp})
applied to a \Mo{} map $g$ and $f_\pm$ shows that the choice of the maps onto $\H$ and $\bL$ does not
affect $S_f$. If $h$ is a conformal map of $D$ onto $\D$ with $h(\zeta)=0$ and $h(a)=\ee^{\ii\theta}$
and $h(b)=-\ee^{-\ii\theta}$, and $g=G_\theta$ as in (\ref{Gc}), then 
by (\ref{Schwarzcomp}), Theorem \ref{unique} and (\ref{SG}) the corollary follows.
\end{proof}

\begin{remark}If $\gamma_{a,\zeta}$ and $\gamma_{\zeta, b}$ are hyperbolic geodesics in $D$ from
$a$ to $\zeta$ and from $\zeta$ to $b$ with tangent directions $\ee^{\ii t}$ and $\ee^{\ii s}$ at $\zeta$, 
then by Theorem \ref{Dgeopair}, the tangent
direction of the geodesic pair in $(D;a,b,\zeta)$ is $\ee^{\ii u}$ at $\zeta$, where $u=(t+s)/2$, and
the corresponding $\theta$ is given by $\theta=(t-s)/2$.
 The density of the hyperbolic metric $\rho(z)|\dd z|$ in $D$ at $z$ is given by $\rho(z)=|h'(z)|$, where $h$ is a conformal map of $D$ onto $\D$. 
Thus, we can write the residue of $S_f$ at $\zeta$ purely in terms of the hyperbolic geometry on $D$ as
$$4\ii\sin\left(\frac{t-s}{2}\right)\rho(\zeta) \ee^{-\ii\left(t+s\right)/2}.$$
\end{remark}

\bigskip

\section{Piecewise geodesic curves}\label{sec:pwcurves}

\subsection{General facts}
If $\gamma=\gamma_1\cup \dots \cup \gamma_n$ is a Jordan curve 
in $\wC$, then 
a welding map
$\go$ for $\gamma$ in $\wC$ is an increasing homeomorphism of 
$\wR=J_1\cup \dots \cup J_n$ into $\wR$. Note that  
$\go\vert_{J_j}$  is a welding map for $\gamma_j$
in $\wC\smallsetminus  (\gamma\smallsetminus \gamma_j)$. 

An increasing homeomorphism $h$ of $\R$ onto $\R$
will be called {\bf piecewise \Mo} if there are $\{x_j\}$ with
$-\infty=x_0 < x_1 < \dots <  x_n =+\infty$ so that $h$ is given by 
a \Mo{} map $L_j$ on $(x_{j-1},x_{j})$, for each $j=1,\dots,n$.
A homeomorphism $h$ of $\R$ onto $\R$ is called {\bf quasisymmetric}
if there is a constant $M < \infty$ such that
\begin{equation}\label{QS}
\dfrac{1}{M}\le \dfrac{|h(x+t)-h(x)|}{|h(x)-h(x-t)|} \le M,
\end{equation}
for all $x,t\in\R$.

A Jordan curve $\gamma\subset \wC$ 
will be called {\bf piecewise geodesic} if
$\gamma=\gamma_1\cup \gamma_2\cup \dots\cup \gamma_n$ where $\gamma_j$ is a hyperbolic geodesic in 
$\wC\smallsetminus (\gamma\smallsetminus \gamma_j)$, for
$j=1,\dots,n$. The points $\zeta_j=\gamma_j\cap \gamma_{j+1}$ will be called the {\it vertices} of $\gamma$, and the $\gamma_j$ the {\it edges}.

\bigskip

\begin{cor}\label{pw}
A piecewise \Mo{} increasing homeomorphism of $\wR$ onto $\wR$ is a welding map associated with a piecewise geodesic 
Jordan curve. Conversely, if
 $\gamma$ is a piecewise geodesic Jordan curve, then 
the welding maps associated with $\gamma$ are piecewise \Mo{} homeomorphisms.
\end{cor}

\begin{proof}
Suppose $\go$ is an  increasing homeomorphism of 
$\wR=J_1\cup \dots \cup J_n$ onto $\wR$ such that 
$\go\vert_{J_j}$ is \Mo. By composing $\go$ with a single increasing \Mo{} map, we may suppose that $\go(\infty)=\infty$. 
Then there is a constant $c<\infty$ such that $1/c<\go'(x)<c$ for all $x\in\R$.
By the mean-value theorem applied to the numerator and denominator of (\ref{QS}), $\go$ is quasisymmetric. By \cite[page 92]{LV} or \cite{LA}, $\go$ is a welding map for
a Jordan curve $\gamma=\gamma_1\cup \dots \cup\gamma_n$.
By Lemma \ref{hypgeo}, 
because
$\go\vert_{J_j}$ is a welding map for $\gamma_j$ in $\wC\smallsetminus
(\gamma\smallsetminus \gamma_j)$, ~$\gamma_j$ is a hyperbolic
geodesic in $\wC\smallsetminus
(\gamma\smallsetminus \gamma_j)$. The converse also follows immediately from Lemma \ref{hypgeo}.
\end{proof}

 A similar result holds for piecewise geodesic Jordan chords in a simply connected region $D\ne\C$. In this case the welding maps are piecewise \Mo{} homeomorphisms of a proper sub-interval in $\wR$.

It is natural to ask if the edges of a piecewise geodesic Jordan curve are hyperbolic geodesics of the punctured sphere, with punctures at the vertices. This is generally not the case:

\begin{lemma}\label{circle}
If  $\g = \cup_{i = 1}^n  \gamma_i$ is a piecewise geodesic Jordan curve and if each edge $\gamma_i$ is a 
%the 
hyperbolic geodesic in the punctured sphere $\Sigma : = \wC \smallsetminus \{\zeta_1, \ldots, \zeta_n\}$ between $\zeta_i$ and $\zeta_{i+1}$, then $\g$ is a circle.
\end{lemma}

\begin{proof}
   For each  $i \in 1, \ldots, n$, $\gamma_i$ is the hyperbolic geodesic in the simply connected domain $\O_+ \cup \O_- \cup \gamma_i$. Hence, there is a antiholomorphic involution $R_i : \O_+ \cup \O_- \cup \gamma_i \to \O_+ \cup \O_- \cup \gamma_i$ where $\gamma_i$ is the set of fixed points.  Both maps $R_i|_{\O_+}$ and $R_i|_{\O_-}$ extend to the boundary and induce two homeomorphisms $\g \to \g$.
   
   We claim that $R_i|_{\O_+}$ and $R_i|_{\O_-}$ also fix all $\zeta_k$, for $k \in 1, \ldots, n$.
   In fact, let $p : \D \to \Sigma$ be the universal cover map. Consider 
   %one fundamental domain 
   a connected component $D$ of $p^{-1} (\O_+ \cup \O_- \cup \gamma_i)$. Its boundary is formed by circular arcs which project to $\g \smallsetminus \gamma_i$ (since $\gamma_k$ are assumed to be hyperbolic geodesics in $\Sigma$), and $p^{-1} (\gamma_i)$ is both a conformal geodesic in $D$ and a circular arc orthogonal to the boundary. Without loss of generality, we assume $p^{-1} (\gamma_i) =  (-1,1)$. Then the involution $R_i$ lifts to an antiholomorphic map $\tilde R_i$ on $D$ which equals the identity on $(-1,1)$ and therefore is the map $z \mapsto \bar{z}$ on $D$. It implies that $\tilde R_i$ maps cusps to cusps. Therefore $R_i$ fixes all $\zeta_k$.
   
   Now if $i \neq j$, then $R_i \circ R_j|_{\O_+}$ is a conformal map $\O_+ \to \O_+$ whose extension to $\g$ fixes all $\zeta_k$, so $R_i \circ R_j|_{\O_+}$ is the identity map. This implies that $R_i = R_j$. We obtain that $R_i$ extends to an antiholomorphic involution $\wC \to \wC$ with set of fixed points $\g$.  Therefore $\g$ is a circle.
\end{proof}

See also the closely related comment after Definition \ref{gp}.

\subsection{$C^1$ piecewise geodesic 
Jordan curves}

Let us first comment on the existence of $C^1$ piecewise geodesic Jordan curves.
In \cite[Prop.\,2.13]{RW} two of the authors introduced the notion of \emph{Loewner energy} for Jordan curves and
showed the existence of minimizers among all curves passing through a given set of vertices. The union of two adjacent edges in those minimizers are geodesic pairs in the complement of the rest of the curve, and do not have logarithmic spirals. By Corollary~\ref{Dlocal}, they are $C^1$ piecewise geodesic Jordan curves\footnote{The proof of the regularity stated in \cite[Prop.\,2.13]{RW} was sketched out and referred to the present work for details.  However, the proof of the existence of the minimizer was self-contained.}.

One of the consequences of the explicit formula for $C^1$ geodesic pairs is the following. 
As before if $\gamma=\gamma_1\cup\dots\cup\gamma_n$ is a piecewise geodesic curve,
let $\zeta_j=\gamma_j\cap\gamma_{j+1}$, $j=1,\dots,n$ where 
$\gamma_{n+1}\equiv\gamma_1$.

%\bigskip

\begin{thm}\label{pwclosed}
Suppose $\gamma\subset\C$ 
is a $C^1$ piecewise geodesic Jordan curve. % with continuous (unit) tangent. 
If $\gO_+, \gO_-$ are the two regions given by $\wC\smallsetminus \gamma$, and if
$f_+, f_-$ are conformal maps onto $\H$ and $\bL$ respectively, set
$f=f_+$ on $\gO_+$ and $f=f_-$ on $\gO_-$. Then the Schwarzian derivative, 
$S_f$, of $f$ extends to be meromorphic in $\wC$ and is given by
\begin{equation}\label{spoles}
S_f(z)=\sum_{j=1}^n \frac{c_j}{z-\zeta_j},
\end{equation}
for some constants $c_j$ satisfying
\begin{equation}\label{sums}
\sum_{j=1}^n c_j=\sum_{j=1}^n c_j \zeta_j =\sum_{j=1}^n c_j \zeta_j^2=0.
\end{equation}
\end{thm}

\begin{proof} By Corollary \ref{SDpair} applied  $\gamma_j\cup\gamma_{j+1}$ in
$D=\wC\smallsetminus (\gamma\smallsetminus(\gamma_j\cup\gamma_{j+1}))$, $j=1,\dots,n$,
$S_f$ is analytic in $\wC$ except for simple poles at each $\zeta_j$. To examine the
behavior of $S_f$ near $\infty$, we can use (\ref{Schwarzcomp}) 
to obtain
$S_{f}(z) = S_{f\circ g}(1/z){z^{-4}}$, where $g(z)=1/z$. 
Because $f\circ g$ is analytic 
at $0$, ~$S_f(z)=O(z^{-4})$ near $\infty$. Thus we can write $S_f$ in the form
(\ref{spoles}), and expanding this form in powers of $1/z$ we obtain (\ref{sums}). 
\end{proof}

The residues $c_j$ are reminiscent to the \emph{accessory parameters} on the $n$-punctured sphere. Similar to \cite{TZ1,TZ2}, we relate $c_j$ to the variation of the Loewner energy in the companion paper \cite{MRW}. 

If $\gamma\subset\C$ is a $C^1$ piecewise geodesic curve, % with continuous tangent, 
then it follows from Theorem
\ref{pwclosed} that
$S_f(z)=q(z)\prod_1^n(z-\zeta_j)^{-1}$, where $q$ is a polynomial of degree $\le n-4$.

\begin{remark}\label{rem:three_point_Sch}
One consequence is that 
$S_f(z)\equiv 0$ if $n \le 3$, 
which implies that $\gamma$ is a circle or a line through the points
$\{\zeta_j\}$. This implies the
curve $\gamma$ is unique when \blue{$n = 3$}.
\end{remark}

If $n=4$ we must have 
\begin{equation}\label{four}
S_f=\frac{C}{(z-\zeta_1)(z-\zeta_2)(z-\zeta_3)(z-\zeta_4)}.
\end{equation}

A version of Theorem \ref{pwclosed} holds for a piecewise geodesic curve $\gamma\subset \wC$ 
with continuous tangent. If $\infty$ is not one of the $\zeta_j$, 
we can use (\ref{Schwarzcomp}) with $g=f$ and \Mo{} $h$ to obtain the same statement as in Theorem
\ref{pwclosed}. If $\infty$ is one of the $\zeta_j$, then again using (\ref{Schwarzcomp})
with $g=f$ and \Mo{} $h$, we obtain (\ref{spoles}) with $n-1$ terms,
one for each finite $\zeta_j$. In this case $S_f=O(z^{-3})$ near $\infty$ so that (\ref{sums}) is replaced by
\begin{equation}
\sum_{j=1}^n c_j=\sum_{j=1}^n c_j \zeta_j =0.
\end{equation}

\bigskip

We were unable to explicitly solve the following.

\bigskip

% \begin{prob} Given four distinct points $\zeta_1,\zeta_2,\zeta_3,\zeta_4 \in \C$. 
% Find the explicit expression of a $C^1$ piecewise geodesic Jordan curve $\gamma$ with vertices $\zeta_1,\zeta_2,\zeta_3,\zeta_4$.
% \end{prob}

\begin{prob}
For which values of $C$ does there exist a piecewise geodesic Jordan curve $\gamma$ with vertices $\zeta_1,\zeta_2,\zeta_3,\zeta_4$
such
that the Schwarzian derivative of the associated conformal maps of
$\wC\smallsetminus \gamma$ onto $\H, \bL$ satisfy (\ref{four})?
\end{prob}

\bigskip

There are several ways to normalize the welding maps.  For the next theorem, the
following normalization is convenient. If $\go$ is a piecewise \Mo{} 
welding map of $\wR$ onto $\wR$ with continuous derivative, then
we can replace $\go$ with $\go\circ \sigma_1$, where $\sigma_1$ is a \Mo{} map so that
$\go$ is analytic on $\wR\smallsetminus \{x_j\}_0^{n}$ with 
$x_0=\infty$, $0=x_1 < x_2 < \dots < x_n=1$.
Moreover, post composing $\go$ with another \Mo{} map, we may suppose that $w(x)=x$ for 
$x < 0$. Because $\go$ has continuous derivative at $\infty$, we then have $\go(x)=x+b$
for some $b\in \R$, when $x > 1$. If $n \ge 2$, then there is exactly one welding map
with this normalization in each equivalance class of piecewise \Mo{} welding maps modulo pre and post
composition by \Mo{} maps.

\bigskip

\begin{thm}\label{weldthm} Suppose $0=x_1< x_2 < \dots < x_n=1$ and $0=y_1< y_2 <
\dots < y_n$. Then
there is a normalized piecewise \Mo{} welding map $\go$ of $\wR$,
analytic on $\R \smallsetminus\{x_1,\dots,x_n\}$, with continuous derivative such that 
\begin{equation*}
\go(\infty)=\infty, \  \go'(\infty)=1, ~{\rm and}~  \go(x_j)=y_j, \  j=1,\dots,n
\end{equation*}
 if and only if
\begin{equation}\label{product}
1= \frac{y_n-y_{n-1}}{x_n-x_{n-1}}\cdot \frac{x_{n-1}-x_{n-2}}{y_{n-1}-y_{n-2}} \cdot
\frac{y_{n-2}-y_{n-3}}{x_{n-2}-x_{n-3}} \cdots %\frac{y_2-y_1}{x_2-x_1}.
\end{equation}
where the last term in (\ref{product}) is $\frac{y_2-y_1}{x_2-x_1}$ when $n$ is even and
$\frac{x_2-x_1}{y_2-y_1}$ when $n$ is odd.
Moreover such a piecewise \Mo{} welding map is unique.
\end{thm}

\begin{remark}
By Corollary \ref{pw}, this theorem gives a parameterization of the family of $C^1$ piecewise 
 geodesic Jordan curves in terms of the real coordinates $(x_2, \ldots x_{n-1}, y_2, \ldots, y_n)$.  Together with the constraint \eqref{product}, this family has real dimension $2n-4$ which coincides with the dimension of the Teichm\"uller space of the $(n+1)$-punctured sphere.  
\end{remark}

\begin{proof}
The proof is based on the following elementary identity which holds for all \Mo{} maps $T$.
If $a\ne b$, then
\begin{equation}\label{identity}
T'(a)T'(b)=\left(\frac{T(b)-T(a)}{b-a}\right)^2.
\end{equation}
Given $a,b,\alpha,\beta, \delta$ then there exist a unique \Mo{} map $T$ such that
$T(a)=\alpha$, $T(b)=\beta$ and $T'(a)=\delta$. Indeed, if
\begin{equation*}
S_1(z)=\frac{\delta(a-b)}{\alpha-\beta}\frac{z-a}{z-b} ~{\rm and}~
S_2(z)=\frac{z-\alpha}{z-\beta},
\end{equation*}
then $T= S_2^{-1}\circ S_1$ works. Moreover if $T_1$ is another such map then
$S=S_2\circ T_1\circ S_1^{-1}$ satisfies $S(0)=0$, $S(\infty)=\infty$ and $S'(0)=1$, so
$S(z)=z$ and $T_1=T$.
The identity (\ref{identity}) allows us to
determine $T'(b)$ from $T(b)-T(a)$, $b-a$ and $T'(a)$.

Suppose $\go$ exists.
Because $\go$ is linear on $(-\infty,0)$ and on $(1,+\infty)$ and 
$\go'(\infty)=1$, we have $\go'(0)=1=\go'(1)$.
By (\ref{identity}) 
\begin{equation}
\go'(x_{j+1})=\left(\frac{\go(x_{j+1})-\go(x_j)}{x_{j+1}-x_j}\right)^2
\frac{1}{\go'(x_j)}.
\end{equation}
So we successively find a formula for $\go'(x_{j+1})$, $j=1,\dots,n-1$, 
starting with $\go'(x_1)=1$.
Because $\go'(1)=1$, by induction (\ref{product}) must hold. Here we note that
individual ratios in (\ref{product}) are all positive, so that the square is not needed
in the statement.
The map $\go$ restricted to each interval
$(x_j,x_{j+1})$, $j=1,\dots,n-1$, is uniquely determined from $\{x_j\}$, $\{y_j\}$ once we know
$\go'(x_j)$. 
Note that we must have $\go(x)=x$, on $(-\infty,0)$ and $\go(x)=x-1+y_n$ on
$(1,+\infty)$, because $\go(0)=0$ and $\go(1)=y_n$.

Conversely, assume (\ref{product}) holds. Define $\go(x)=x$, for $x < 0$. 
Then $\go'(0)=1$, so we can uniquely define
$\go$ on each subsequent interval because of (\ref{identity}). Because (\ref{product}) holds, 
we obtain $\go'(1)=1$. Then set $\go(x)=x-1+y_n$ for $ x > 1$.
\end{proof}

If we do not assume that $\go'(\infty)=1$ then the statement of the corresponding result
is somewhat more complicated. We must split the result into two cases, when $n$ is
even and when $n$ is odd. When $n$ is even, there is only one possible derivative at
$\infty$, and it is given in terms of the product. 
When $n$ is odd, any choice of the derivative at $\infty$ will produce a
welding map for these intervals, if and only if (\ref{product}) holds. 
In the latter case, once the derivative at $\infty$ is
chosen, the welding map for the set of intervals is unique.
We are interested, though,  in $C^1$ piecewise geodesic curves. 
If $\go$ is a welding map for a Jordan curve $\gamma$ with respect to a region $\gO$,
then so is the normalization of $\go$. So characterizing the normalized maps gives a characterization of the equivalence classes
of all piecewise \Mo{} welding maps with continuous derivatives, modulo pre and post composition by \Mo{} maps.

We obtain another proof of Remark~\ref{rem:three_point_Sch}:
\begin{cor}\label{cor:three_points}
A $C^1$ piecewise geodesic Jordan curve with three edges is a circle.
\end{cor}
\begin{proof}
By Corollary~\ref{pw}, the welding homeomorphism $\go$ of a $C^1$ piecewise geodesic Jordan curve with three edges can be chosen to be piecewise M\"obius as in Theorem~\ref{weldthm} with $n = 2$, $0 = x_1 < x_2 = 1$, $0 = y_1 < y_2$, and $(y_2 -y_1)/ (x_2-x_1) = 1$. It implies that $y_2 = x_2 = 1$. Therefore $\go$ is the identity map on  $\wR$ which implies that the welded curve is a circle.
\end{proof}

\bigskip

Our explicit construction of geodesic pairs also gives a little more information about the 
behavior of a piecewise geodesic curve
$\gamma=\gamma_1\cup\dots\gamma_n$, $n\ge 3$, near a point $\zeta_j=\gamma_j\cap\gamma_{j+1}$.
Applying a conformal map of $\wC\smallsetminus (\gamma\smallsetminus(\gamma_j\cup\gamma_{j+1}))$ onto
$\D$, it is enough to investigate the behavior of geodesic pairs in
$(\D;\ee^{\ii\theta},-\ee^{-\ii\theta},0)$ near $0$.
If $\gamma$ is a Jordan curve with continuous tangent at
$\zeta$ then the conformal maps $f_\pm$ of the two regions in 
$\wC\smallsetminus \gamma$ onto $\H$ and $\bL$ respectively are
$C^{1-\eps}$ for $\eps>0$, but not necessarily in $C^1$. See \cite{GM}, Section II.4.

\begin{cor} \label{smoothness}
A $C^1$ piecewise geodesic Jordan curve %with continuous (unit) tangent 
is piecewise analytic and in
$C^{2-\gve}$ for all $\gve > 0$.
\end{cor}

\begin{proof} Piecewise geodesic curves are piecewise analytic by definition. To show the regularity near vertices, it is enough to prove the corollary for geodesic pairs
in $(D;\ee^{\ii\theta},-\ee^{-\ii\theta},0)$.
Fix $\theta$ and set $G=G_\theta$ as in (\ref{Gc}) and let
$\gamma=\gamma_1\cup\gamma_2$ be the corresponding geodesic pair in  $\D$.
Set $H(z)=1/G(z)$. Then  $H$  maps $\gamma$
onto the real axis with $H(0)=0$. Then
\begin{equation}
H'(z)= \frac{2(1+2ic z-z^2)}{(1-2ic z\log z +z^2)^2}= 2 + 8ic z \log z
+{\rm O}(z).
\end{equation}

Thus $(H'(z)-H'(0))/z^{1-\gve}$ is continuous at $0$, for $\gve > 0$.
By Lemma II.4.4 in \cite{GM}, $\gamma$ is in
$C^{2-\gve}$, for $\gve>0$. 

This result can also be deduced using a result in 
\cite{wong} on Loewner slits.
\end{proof}

We remark that Corollary \ref{smoothness} implies that a piecewise \Mo{} welding map $\go$ in
$C^1$ is in 
$C^{2-\eps}$, a fact that can more easily be deduced directly from its definition.

\subsection{Piecewise geodesic Jordan curves with spirals}

Now we turn to general piecewise geodesic Jordan curves $\gamma$. From Corollaries~\ref{local} and \ref{Dlocal}, $\gamma$ has a logarithmic spiral at the vertices  where $\gamma$ does not have continuous (unit)
tangent. %, we return to our discussion of welding maps for the logarithmic spirals $S_\theta$.
\begin{definition}\label{defspiralrate}
If $\gamma(t)=\zeta_0 +\ee^{r(t)+\ii\varphi(t)}$ is a spiral with ``eye'' at $\zeta_0$, then the {\bf spiral rate} of $\gamma$ at $\zeta_0$ is given by
\begin{equation}\label{spiralrate} 
R(\gamma,\zeta_0)=\lim_{r\to 0} \dfrac{\partial \varphi}{\partial r}
\end{equation}
\end{definition}
So $\gamma$ has a large spiral rate at $\zeta_0$ if $|\arg \gamma(t)|$ increases rapidly as $\gamma(t) \to \zeta_0$, and it has a small spiral rate if $|\arg \gamma(t)|$ increases slowly as $\gamma(t)\to \zeta_0$. If $R(\gamma,\zeta_0)$ is positive, then $\gamma$ spirals clockwise as $\gamma(t)\to \zeta_0$. If $R(\gamma,\zeta_0)$ is negative, then $\gamma$ spirals counter-clockwise as $\gamma(t)\to \zeta_0$.

If $S_\theta$ is the spiral given in Example \ref{nonsmooth}, then we can write
$S_\theta(t)=e^{r(t)+\ii\varphi(t)}$ where
$$r(t)=\cos^2{\theta}(t-\pi \tan \theta)$$
and
$$\varphi(t)=\cos^2{\theta}(\pi+t\tan{\theta}),$$
so that the spiral rate at $0$ is given by
\begin{equation}\label{R(S)}R(S_\theta,0)=\tan{\theta}.
\end{equation}
Recall that an associated welding map is given by (\ref{spiralweld}) with $a=\exp(-2\pi \tan \theta)$. Thus
\begin{equation}
\label{eq:R_derivative}
R(S_\theta,0)=\dfrac{1}{2\pi}
\log \dfrac{\om'(0^+)}{\om'(0^-)}.
\end{equation}

This allows us to compute the jump in the derivative of the welding map for a piecewise geodesic curve in terms of the geometry of the curve, as follows.

\begin{lemma}\label{jump}
Suppose $\gamma=\gamma_1\cup\dots\cup\gamma_n$ is a piecewise geodesic Jordan curve,  with piecewise \Mo{} welding map $\go$.
Set $\zeta_j=\gamma_j\cap\gamma_{j+1}$. 
If $x_j \in \partial \H$ corresponds to $\zeta_j \in \gamma$, then $\go'(x)$ has one-sided
limits $\go'(x_j^+)$ and $\go_j'(x_j^-)$ at $x_j$. 
The ``jump''
in the derivative at $x_j$ is given by
$$\frac{\go'(x_j^+)}{\go_j'(x_j^-)}=e^{2\pi R(\gamma,\zeta_j)}.$$
\end{lemma}
\begin{proof} Set $1/a= \frac{\go'(x_j^+)}{\go_j'(x_j^-)}$. By the chain rule, $a$ does not change if we apply a \Mo{} map of $\H$ onto $\H$, so that $a$ does not depend on our choice of the welding map for $\gamma$. If $F$ is conformal near $\zeta_j$ then the spiral rate of $F(\gamma)$ at $F(\zeta_j)$ is the same as the spiral rate of $\gamma$ at $\zeta_j$.
If $a\ne 1$ then by Corollary \ref{local} and Example \ref{weldnonsmooth},
$\gamma$ is the conformal image of $S_\theta$ near
$\zeta_j$ where $\log a = -2\pi \tan \theta$. The lemma then follows from (\ref{R(S)}) and  \eqref{eq:R_derivative}.
\end{proof}

If $x_j=\infty$, then the jump is $\lim_{x\to +\infty}\omega(x)/\lim_{x\to -\infty} \omega(x)$, as can be seen by replacing $\omega$ with $\varphi\circ\omega\circ\varphi^{-1}$ where $\varphi(z)=1/(1-z)$, then computing the jump at $0=\varphi(\infty)$.

We remark that for a piecewise geodesic curve $\gamma$,  because $\dfrac{\partial \varphi}{\partial r}$ is constant for
the curve $S_\theta$, the spiral rate is also equal to
\begin{equation}\label{sr2}R(\gamma,\zeta_j)=\lim \dfrac{2\pi}{\log \dfrac{|\alpha-\zeta_j|}{|\beta-\zeta_j|}},
\end{equation}
where the limit is taken over $\alpha,\beta \ni \gamma \to \zeta_j$ such that $\arg (\alpha-\zeta_j)=2\pi(k+1)$ and $\arg(\beta-\zeta_j)=2\pi k$. This gives the spiral rate in terms of the asymptotic decay rate of $|\gamma-\zeta_j|$ as $\gamma$ loops once around $\zeta_j$.

A similar discussion of the relationship between the jump in the welding map and the spiral rate can be found in \cite{KNS}, although with a different definition of the spiral rate.

\begin{figure}
    \centering
    \includegraphics [width=.5\textwidth, angle=-90]{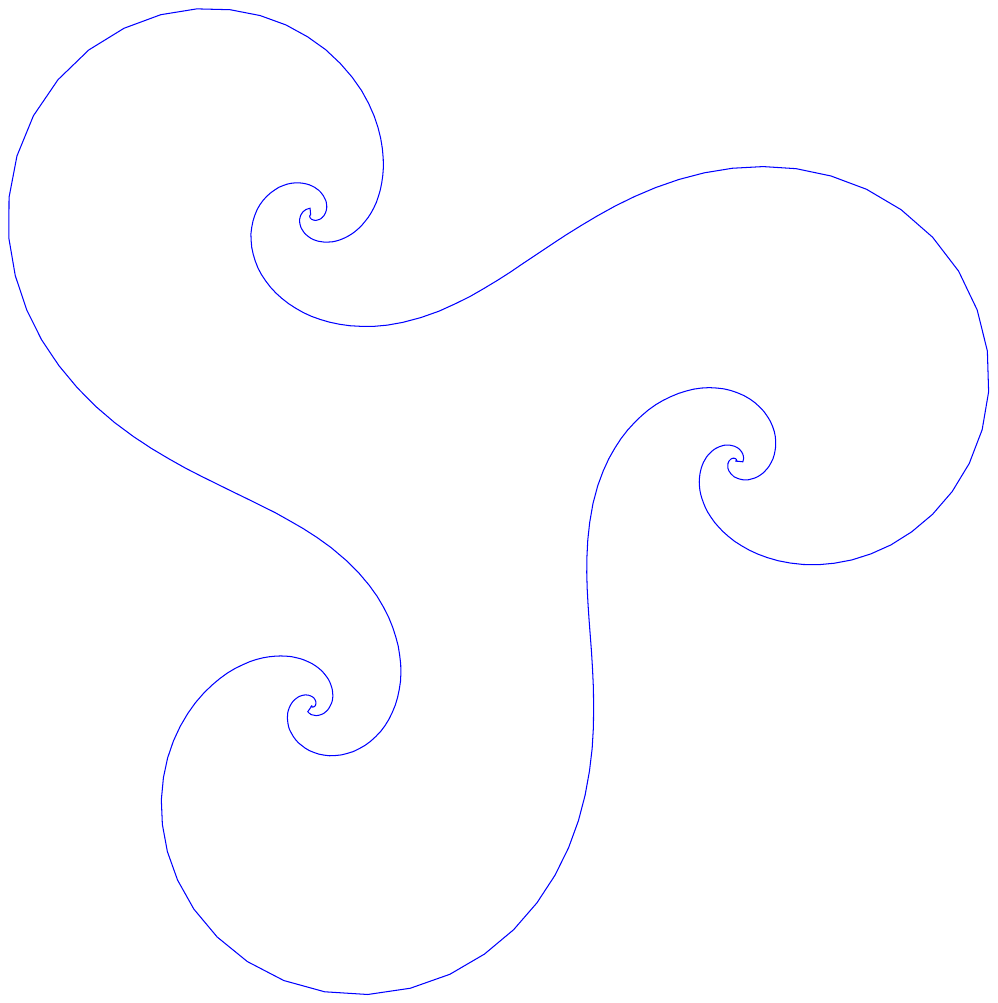}
    \caption{\label{3spirals} A piecewise geodesic curve consisting of three spirals with spiral rate 2, reminiscent of the Triskele, one of the oldest Irish Celtic symbols.}
\end{figure} 

\bigskip

\begin{thm}\label{schwarzian}
Suppose $\gamma\subset\wC$ 
is a piecewise geodesic Jordan curve, analytic except at
$\zeta_1,\dots, \zeta_n\in \C$. 
If $\gO_+, \gO_-$ are the two regions given by $\wC\smallsetminus \gamma$, and if
$f_+, f_-$ are conformal maps onto $\H$ and $\bL$ respectively, set
$f=f_+$ on $\gO_+$ and $f=f_-$ on $\gO_-$. Let $R_j$ be the geometric spiral rate
of $\gamma$ at $\zeta_j$ given in
(\ref{spiralrate}) and (\ref{sr2}).
Then the Schwarzian derivative, 
$S_f$, of $f$ extends to be meromorphic in $\wC$ and is given by
$$S_f(z)=\sum_{j=1}^n 
\frac{\half R_j^2+\ii R_j}{(z-\zeta_j)^2}+ \frac{c_j}{z-\zeta_j},$$
for some constants $c_j$. Near $\infty$, $S_f= O(z^{-4})$.
\end{thm}

\begin{proof}
By Corollary \ref{SDpair}, if $\gamma$ has a continuous tangent at $\zeta_j$ then $S_f$ has a simple pole at
$\zeta_j$. If $\gamma$ does not have a continuous tangent at $\zeta_j$ then to examine the local behavior
of $S_f$ near $\zeta_j$, 
by Lemma \ref{local}, (14), and Example \ref{nonsmooth}
we can write $f= g \circ \varphi$
where $\varphi$ is analytic and one-to-one in a neighborhood of $\zeta_j$ with $\varphi(\zeta_j)=0$ and $g(z)=e^{(1-\ii\tan \theta)\log z}$, 
for the appropriate $\theta$.
Then $S_g(z)=\lambda z^{-2}$, where $\lambda= \ii \tan\theta +\half (\tan \theta)^2$.
By the identity (\ref{Schwarzcomp}), 
$S_{g\circ \varphi}(z)=\frac{\lambda}{(z-\zeta_j)^2}+\dots$ has a double pole
at $\zeta_j$ with the same coefficient $\lambda$. By (\ref{R(S)}),
this proves the expansion stated in the theorem at points where the curve is not continuously differentiable.
%smooth.
\end{proof}

Corollary \ref{smoothness} and Theorem \ref{schwarzian}
yield the dichotomy that a piecewise geodesic curve is either $C^{2-\varepsilon}$ for every $\varepsilon >0$ at a vertex or a logarithmic spiral at a vertex.

A welding map of the form (\ref{pwMo}) with $a_1\ne a_2$ can be conjugated using linear maps to an equivalent welding of the form
\begin{equation}
\go(x)=\begin{cases}
x & \text{for}~ x > c_1,\cr
ax & \text{for}~ x < c_2,\cr
\end{cases} 
\end{equation}
with $c_2 \le c_1$. If $c_1 >0$ and $c_2 < 0$, 
we can then use the examples created earlier for welding maps $h$ of the form
$h(x)=ax$, $x>0$ and $h(x)=x$, when $x <0$ to create geodesic pair (for each $a>0$) in
$\C$ which meet at $0$ and $\infty$ and form logarithmic spirals at $0$ and $\infty$. We reduce
the size of the intervals $x >0$ and $x < 0$ by removing portions of the two spirals from $\C$
to obtain a simply connected region and a geodesic pair with the desired welding map. See Corollary \ref{local}.
(The explicit form of a conformal map from the upper half-plane to the complement in $\C$
of a portion of $S_\theta$ is given in 
\cite{LMR}, Proposition 3.3).   In this way, we
obtain geodesic pairs for all welding maps of the form
\blue{$\go(x)=x$, for $x > c_1$, and $\go(x)=ax$ for $x < c_2$}, provided
$c_2 \le 0 \le c_1$, and of course any welding map conjugate to these.

The maps for which we do not have an explicit construction are when 
\blue{$0 < c_2 \le c_1$ or $ c_2 \le c_1 < 0$}. 
The point is that Example \ref{nonsmooth}
creates an infinite spiral at points corresponding to both $0$ and
$\infty$, the two fixed points of the maps $h(x)=ax$ and $h(x)=x$. 
These fixed points
cannot be interior to the intervals where the welding map is the identity, or where it is
given by $ax$ because the associated maps $f_+^{-1}$ and $f_-^{-1}$ must be analytic where the welding is M\"obius.

\bigskip

\begin{prob} Analogous to \eqref{Gc},
if $a < 1$, find an expression for a conformal map from $\D\smallsetminus\gamma$
onto $\H\cup\bL$, where $\gamma=\gamma_1\cup\gamma_2$ is a geodesic pair
in $\D$, such that the welding
map is given by $h(x)=ax$ for $x <1$ and $h(x)=x$ for $ x > 2$ onto $\mathbb H$. 
\end{prob}

%One might also ask for geodesic pairs in the simply 5connected region
%$\C$. In this case, because the welding map must map %$\wR$ onto
%$\wR$, the welding is conjugate to $h_a(x)=x$ for $x < 0$ %and
%$h_a(x)=ax$ for $x > 0$, $a > 0$. This of course gives %the logarithmic
%spirals
%discussed earlier, if $a\ne 1$, and a straight line or %circle if $a=1$.

\section{$C^1$ piecewise geodesic curves with four edges}\label{sec:fourlegs}

As we noted in Remark~\ref{rem:three_point_Sch} and Corollary~\ref{cor:three_points}, a $C^1$ piecewise geodesic curve with three edges must be a circle. We now consider $C^1$ piecewise geodesic curves with four edges.

The conformal modulus of a Jordan domain marked by four consecutive boundary points can be defined as the aspect ratio of a rectangle that arises as the image of the domain under a conformal map sending the marked points to the corners. Our description of piecewise \Mo{} weldings Theorem \ref{weldthm} easily gives the following symmetry:

\begin{lemma}\label{modulus} If $\gamma=\gamma_1\cup \dots \cup \gamma_4$ is a $C^1$ %smooth
 piecewise geodesic curve, then the conformal moduli of the two complementary domains $\gO_+$ and $\gO_-$, marked by the vertices $\zeta_1,...,\zeta_4$ of $\gamma$, coincide.
\end{lemma}

\begin{proof} 
By Corollary \ref{pw}, the welding maps associated with $\gamma$ are piecewise \Mo{} homeomorphisms. The conformal maps $f_+:\gO_+\to\H$ and $f_-:\gO_-\to\bL$ can be chosen so that $f_+(\zeta_4)=f_-(\zeta_4)=\infty$ and such that
the welding $\go=f_- \circ {f_+}^{-1}$ satisfies the normalization of Theorem \ref{weldthm} with $x_j=f_+(\zeta_j), y_j=f_-(\zeta_j), j=1,2,3$, \blue{where $x_1=y_1=0$ and $x_3=1$.} Then \eqref{product} gives
$$1 = \frac{y_3-y_2}{x_3-x_2}\cdot \frac{x_2-x_1}{y_2-y_1}\blue{=\frac{y_3-y_2}{1-x_2}\cdot\frac{x_2}{y_2}}$$
and it follows that $x_2=y_2/y_3$. Thus, the upper half-plane, marked by $\infty, 0, x_2, 1$ is anti-conformally equivalent to the lower half-plane
marked by $\infty, 0, y_2, y_3.$
\end{proof}

A special feature of four edge piecewise geodesic curves is that they have more symmetries, namely the invariance under the four conformal automorphisms of $\wC\smallsetminus\{\zeta_1,\zeta_2,\zeta_3,\zeta_4\}$
 that permute the vertices:

\begin{thm}\label{aut}If $\gamma$ is a $C^1$
piecewise geodesic curve and if $T$ is a conformal automorphism of $\wC\smallsetminus\{\zeta_1,\zeta_2,\zeta_3,\zeta_4\}$, then $T(\gamma)=\gamma.$ Moreover, if $T$ is the identity or if $T(\zeta_1)=\zeta_3$ and $T(\zeta_2)=\zeta_4,$ then $T(\gO_+)=\gO_+$ and $T(\gO_-)=\gO_-$,
whereas the other two automorphisms interchange $\gO_+$ and $\gO_-$.
\end{thm}

\begin{proof}
We have described earlier (see the discussion after Definition \ref{equiv})
that equivalent homeomorphisms of the real line arise as welding homeomorphisms of the same curve.
However, if the homeomorphism and the curve are suitably normalized, then the welding uniquely determines
the curve and the (suitably normalized) conformal maps representing it. 
Our proof is based on this uniqueness, together with our characterization of %smooth 
$C^1$ piecewise M\"obius
homeomorphisms.

We may assume that the four points are $0,\zeta,1, \infty$ and that the two conformal maps $f_+$ and $f_-$ are normalized to fix $0,1,\infty.$ Then the welding homeomorphism $\omega=f_-\circ f_+^{-1}$ fixes $0,1, \infty$ and by Lemma \ref{modulus} also fixes $x=f_+(\zeta)\in (0,1).$
Note that $d=\omega'(\infty)>0$ but does not necessarily equal 1.
By Theorem \ref{weldthm} (and the discussion \blue{following its proof}), $\omega$ is the unique %smooth 
$C^1$ piecewise \Mo{} homeomorphism fixing
$0,x,1,\infty$ with derivative $d$ at $\infty.$

If $T$ is the automorphism interchanging $0,\infty$ and $\zeta,1$, namely
$T(z)=\zeta/z,$ let $S(z)=x/z$ and notice that $S$ induces the same permutation on $\{0,x,1,\infty\}$ as $T$ on $\{0,\zeta,1,\infty\},$
and that $S$ interchanges $\H$ and $\bL.$ 
The conformal maps $g_+=T\circ f_-^{-1}\circ S$ of $\H$ and $g_-=T\circ f_+^{-1}\circ S$ of $\bL$ solve the same normalized welding problem as 
$f_+^{-1}$ and $f_-^{-1}$: Indeed, the homeomorphism 
$$g_-^{-1}\circ g_+=S^{-1}\circ f_+ \circ f_-^{-1} \circ S= S\circ \omega^{-1}\circ S$$
of $\wR$ is piecewise \Mo{} with continuous derivative, fixes ${0,x,1,\infty}$, and has the same derivative as $\omega$ at $\infty$ by a simple computation (using $\omega^\prime(0)=d$ since $\omega$ is linear on the interval $(-\infty,0)$).
It follows that $g_+=f_+^{-1}$ so that $g_+(\H)=T f_-^{-1}(\bL)$ and therefore $T(\O_+)=\O_-$ and $T(\gamma)=\gamma.$
%$\g=T(\g).$ 
The same argument applies when $T$ interchanges $0,\zeta$ and $1,\infty$.

If $T$ interchanges $0,1$ and $\zeta,\infty$, the corresponding \Mo{} transformation $S$ that interchanges $0,1$ and $x,\infty$ fixes $\H$ and $\bL,$
and the above reasoning applied to $g_+=T\circ f_+^{-1}\circ S$ on $\H$ and $g_-=T\circ f_-^{-1}\circ S$ on $\bL$ again yields $g_+=f_+^{-1}$, hence
$T(\O_+)=\O_+.$
\end{proof}

An alternative proof, again assuming without loss of generality that the four points are $0$, $\zeta$, $1$, and $\infty$, goes as follows: The maps $\tau_1(z)=\zeta/z$, $\tau_2(z)=(z-\zeta)/(z-1)$, and $\tau_3(z)=\zeta(z-1)/(z-\zeta)=\tau_1\circ\tau_2(z)$ each preserve the set $P=\{0,\zeta,1,\infty\}$, so that if $\gamma$ is a piecewise geodesic Jordan curve which is analytic except at these points, then
$\tau_j^{-1}(\gamma)$ is also a piecewise geodesic Jordan curve which is analytic except at these points. Using the notation of Theorem \ref{pwclosed}, if $f$ maps the two regions determined by $\gamma$ onto $\H \cup \bL$, then $f\circ \tau_j$ maps the two regions determined by $\tau_j^{-1}(\gamma)$ onto $\H \cup \bL$. By Theorem  \ref{pwclosed}, (\ref{Schwarzcomp}), and (\ref{four})
\begin{equation}
    S_{f\circ \tau_j}(z)=S_f(\tau_j(z))(\tau_j^\prime(z))^2=S_f(z).
\end{equation}
Thus there are \Mo{} maps $\sigma_j$ so that
\begin{equation}
f\circ \tau_j=\sigma_j\circ f.
\end{equation}
But $\Im f \ne 0$ if and only if
$z \notin \gamma$, and similarly $\Im f\circ\tau_j \ne 0$ if and only if $z \notin \tau_j^{-1}(\gamma)$. Because $\tau_j(P)=P$, there must be at least $4$ distinct points in $\sigma_j(\wR)\cap \wR$. Because $\sigma(\wR)$ is either a circle or a line, $\sigma_j$ must map $\wR$ onto $\wR$, and hence $\tau_j^{-1}(\gamma)=\gamma$. This equality is an equality of sets. The orientation can be reversed.
Moreover $\tau_j(\Om_+\cup\Om_-)=\Om_+\cup \Om_-$. The last statement in the theorem follows because we can compute the direction  of $\tau_j^{-1}(\gamma)$ by the order of the image points
$\tau_j^{-1}(0), \tau_j^{-1}(\zeta), \tau_j^{-1}(1), \tau_j^{-1}(\infty)$.

\end{document}